\documentclass[reqno, a4paper]{amsart}
\usepackage{amssymb, amscd, amsfonts,  amsthm}
\usepackage[mathscr]{eucal}
\usepackage{graphicx}
\usepackage[all,poly]{xy} %%%%for leo's diagram and polyhedra
 \usepackage[all]{xy} 
 
\newtheorem{theorem}{Theorem}[section]
\newtheorem{lemma}[theorem]{Lemma}
\newtheorem{corollary}[theorem]{Corollary}
\newtheorem{proposition}[theorem]{Proposition}

\theoremstyle{definition}

\newtheorem{remark}[theorem]{Remark}

\newcommand{\restrict}{\,{\mathbin{\vert\mkern-0.3mu\grave{}}}\,}

\newcommand{\luk}{\L u\-ka\-s\-ie\-wicz}

\newcommand{\remove}[1]{}

\DeclareMathOperator{\Rn}{{\mathbb R^{\it n}}}

\DeclareMathOperator{\maxspec}{\boldsymbol{\mu}}
\DeclareMathOperator{\Maxspec}{{\rm Maxspec}}
\DeclareMathOperator{\spec}{{\rm Spec}}
\DeclareMathOperator{\McN}{\mathcal M}

\DeclareMathOperator{\McNn}{\mathcal M([0,1]^{\it n})}

\DeclareMathOperator{\interval}{[0,1]}
\DeclareMathOperator{\cube}{[0,1]^{\it n}}

\DeclareMathOperator{\range}{\rm range}

\DeclareMathOperator{\den}{\rm den}

 \title[Hopfian MV-algebras, unital $\ell$-groups, AF C$^*$-algebras]
{Hopfian $\ell$-groups, MV-algebras and AF C$^*$-algebras}

\author{Daniele Mundici}

\address[D. Mundici]{Department of
Mathematics and Computer Science  ``Ulisse Dini'' \\
University of Florence\\
Viale Morgagni 67/A \\
I-50134 Florence \\
Italy}
\email{ mundici@math.unifi.it }

\date{\today}

\begin{document}



\begin{abstract}
An algebra  is said to be {\it hopfian\/} 
if it is not isomorphic to a proper
quotient of itself. 
We describe several classes of 
hopfian and of non-hopfian   
unital lattice-ordered abelian groups
and MV-algebras. Using Elliott classification and
$K_0$-theory,  we apply our results to
 other related structures, notably
the Farey-Stern-Brocot AF C$^*$-algebra and all its primitive
quotients, including 
%the Effros-Shen
%C$^*$-algebras  $\mathfrak F_\theta$, 
the  Behnke-Leptin  C$^*$-algebras  $\mathcal A_{k,q}$. 
\end{abstract}

\maketitle

%%%%%%%%%%%%%%%%%%%%%%%%

\section{Introduction}
Since the publication of 
 \cite{hop} and \cite{mal} the literature on hopfian
 algebras and spaces has expanded rapidly.
 The survey
  \cite{var1} may give an idea of the applicability of
  the notion of hopficity and of the various methods used
  to prove that   structures have or do not
  have the hopfian property.

Also  the literature on the  equivalent categories of MV-algebras
 and unital $\ell$-groups
 has been steadily expanding over the last
 thirty years. See 
 \cite{buscabmun, cab-ja, cab-forum, 
 cabmun,  cabmun-ccm,  carrus, gehgoomar}
 for a selection of
  recent papers, and the monographs  \cite{cigdotmun, mun11}
  for a detailed account on the relationships between these
  structures and rational polyhedra, AF C$^*$-algebras, 
  Grothendieck topoi, Riesz spaces, multisets, etc.

 Remarkably enough, the literature on hopfian MV-algebras, 
 (unital as well as non-unital)   $\ell$-groups and
 C$^*$-algebras is virtually nonexistent.
 The aim of this paper is to give a first account
 of the depth and  multiform beauty  of
 this theory,  with its own  geometric, algebraic,
 and  topological   techniques.  We will apply our results
 to the Farey-Stern-Brocot AF C$^*$-algebra $\mathfrak M_1$,
 \cite{mun-adv,boc}, 
 and all its quotients, including 
% the Effros-Shen
%C$^*$-algebras  $\mathfrak F_\theta$, \cite{eff}, and
the  Behnke-Leptin  C$^*$-algebras  $\mathcal A_{k,q}$,
\cite{behlep}.   
 
 \smallskip
 
 We recall that 
 a {\it unital $\ell$-group} is an abelian
 group $G$ equipped with a translation invariant
 lattice-order and a distinguished {\it archimedean}
  element $u$, called the (strong, order) {\it unit}.
  In other words,  every $0\leq x\in G$ is dominated by
 some integer multiple of $u$. A {\it homomorphism}   
 of unital $\ell$-groups preserves the unit as well as the group and
 the lattice structure.
 
 An {\it MV-algebra} is an involutive
 abelian monoid $A=(A,0,\neg,\oplus)$
 satisfying the equations  $x\oplus \neg 0=\neg 0$ and
 $\neg(\neg x\oplus y)\oplus y = 
 \neg(\neg y \oplus x)\oplus x.$
 Equivalently, by Chang completeness theorem 
 \cite[Theorem 2.5.3]{cigdotmun},  
 $A$ satisfies all
 equations satisfied by the real unit interval $[0,1]$
 equipped with the operations $\neg x=1-x$ and
 $x\oplus y=\min(1,x+y)$.
 Boolean algebras coincide with MV-algebras
 satisfying the equation $x\oplus x=x.$

  In
  \cite[Theorem 3.9]{mun-jfa} a natural  equivalence
  $\Gamma$ is established between 
  unital $\ell$-groups and MV-algebras.

\smallskip
A  (universal) algebra $R$
 is {\it residually finite}  if for any $x \not= y \in R$ there
is a homomorphism $h$  of $R$ 
into a finite algebra such that $h(x)\not=h(y)$.
%In other words, the congruences  $\theta$
%with finite quotient  $R/\theta$ have the trivial 
%congruence $\{(x,x)\mid x\in R\}$ as their intersection.
This general notion is investigated 
 in Malcev's book  \cite[p. 60]{mal-book},
with the name ``finite approximability''.

\smallskip
The following result generalizes 
a  group-theoretic theorem  due to  Malcev
\cite{mal}:
\begin{theorem}
\cite[Theorem 1]{eva}, 
\cite[Lemma 6, p. 287]{mal-book}  %  \cite[p. 167]{blaneu}
\label{theorem:evans}
\,\,Every finitely  generated residually finite
algebra $R$ is  {\em hopfian}, meaning that
every 
homomorphism  of   $R$ onto $R$ is 
an automorphism.
\end{theorem}

In Theorem \ref{theorem:residually-finite}
finitely generated residually finite MV-algebras
will be shown to coincide
with semisimple MV-algebras
whose finite rank
maximal ideals  are dense in their maximal spectral space.

Then as a particular case of Theorem \ref{theorem:evans}, 
 every finitely generated semisimple MV-algebra 
whose finite rank
maximal ideals  are dense  is hopfian, (Corollary \ref{corollary:epi}).
Corollary \ref{corollary:epi-bis} yields the counterpart
of this result for  unital $\ell$-groups: 
if $(G,u)$ is finitely generated
and semisimple, and the finite rank maximal ideals of $G$
are dense in the maximal spectral space of $G$, then
$(G,u)$ is hopfian.

%In the category of 
%unital $\ell$-groups,
%epimorphisms form a strictly larger class than
%surjective  homomorphisms, i.e.,  regular unital epimorphisms, 
% \cite[p.1321]{cab-forum}.
%Similarly, by  Theorem \ref{theorem:gamma},
% in the  category of MV-algebras,
%epimorphisms form a strictly larger class than
%surjective homomorphisms.
%%
%Thus   
%Theorem \ref{theorem:epi} is independent of (and stronger than)
%the MV-algebraic case of
%Theorem \ref{theorem:evans}. 

In Section \ref{section:other} we show that
several classes of MV-algebras have the
hopfian property. This is the case of 
finitely presented MV-algebras,
finitely generated projective, and in particular,
finitely generated free MV-algebras.
Further,   finitely
 generated MV-algebras with  finite prime spectrum
 are  hopfian, as well as MV-algebras whose
 maximal spectral space is a manifold without boundary,
 (Corollary \ref{corollary:hopfians} and Theorems
 \ref{theorem:germ}, 
\ref{theorem:manifold}).
Theorem \ref{theorem:non-hopf} provides notable
examples of non-hopfian MV-algebras.
Corollary \ref{corollary:seven} yields hopfian and
non-hopfian examples of MV-algebras satisfying any
two of the three conditions of being finitely generated,
  semisimple, and having a dense set of finite rank maximal ideals,
   along with the negation of the third
 condition.
 
Each MV-algebraic result has an equivalent
counterpart for unital $\ell$-groups, via the categorical equivalence
$\Gamma$  (see Corollaries
\ref{corollary:epi-bis},
\ref{corollary:finpres-group},
\ref{corollary:germ-group},
\ref{corollary:nonhopf-group},
\ref{corollary:seven}).  
%Preservation properties are touched upon  
%in Proposition \ref{proposition:preservation}.

Sections \ref{section:applications1}
and \ref{section:applications2}  
are devoted to applications outside the domain of
MV-algebras and unital $\ell$-groups. In 
Corollary 
\ref{corollary:hopfian-freel}   finitely generated free  $\ell$-groups
(without a distinguished unit) are shown to be hopfian.  
Let $\mathfrak M_1$ be the 
Farey-Stern-Brocot  AF C$^*$-algebra  introduced in 1988, 
 \cite{mun-adv},  
  and recently rediscovered and
  renamed $\mathfrak A$, 
\cite{boc, eck}.
 Theorem \ref{theorem:res-fin-dim} shows that
$\mathfrak M_1$
has a separating family of finite dimensional representations,
(i.e., $\mathfrak M_1$ is {\it residually finite dimensional\/}).
%, \cite{goomen}).
Composing  the $\Gamma$ functor with  Grothendieck
$K_0$, in Theorem
\ref{theorem:af-hopf} we show that $\mathfrak M_1$
is hopfian. The proof relies on  
Elliott classification and Bott periodicity theorem, \cite{eff}.
As shown in Corollary \ref{corollary:quotients},  the
 hopfian property is inherited
by all primitive quotients of $\mathfrak M_1$, 
including 
%the Effros-Shen
%C$^*$-algebras  $\mathfrak F_\theta$  and
the  Behnke-Leptin  C$^*$-algebras  $\mathcal A_{k,q}$.

We refer to \cite{cigdotmun, mun11} for background on
 MV-algebras,  to  \cite{andfei} and \cite{bkw}
 for unital $\ell$-groups, and to \cite{eff} for
 AF C$^*$-algebras, $K_0$ and Elliott classification.  Unless otherwise
specified, all MV-algebras and all unital $\ell$-groups in this
 paper are nontrivial.  To help the reader,
 in Section \ref{section:appendix} we record the
 most basic MV-algebraic theorems used 
throughout Sections \ref{section:two}-\ref{section:applications2}. 
 
\noindent 
%%%%%%%%%%%%%%%%%%%%%%%%%%
\section{Characterizing residually finite MV-algebras and unital $\ell$-groups}
\label{section:two}
%%%%%%%%%%%%%%%%%%%%%%%%

{\it The spectral topology of an MV-algebra, \cite[\S 4]{mun11}.}
Unless otherwise specified, all ideals considered in this paper
are proper.  For any MV-algebra $A$ we  let $\spec(A)$ denote
 the space  of its
{\it prime} ideals---those ideals $\mathfrak p$ of $A$ such that
$A/\mathfrak p$ is an MV-chain. 
Following   \cite[Definition 4.14]{mun11},  
$\spec(A)$ comes equipped
 with the  {\it spectral}, or {\it hull kernel}
  topology: its closed sets have the form  
$\mathsf{F}_\mathfrak j = \{\mathfrak p\in \spec(A)\mid
\mathfrak p\supseteq \mathfrak j\}$, 
where  $\mathfrak j$  ranges over
ideals of $A$, plus the trivial ideal $A$.
The resulting topological space  is known as
the {\it (prime) spectral space} of $A$.

\smallskip
The  {\it (hull kernel)  topology} of the {\it maximal spectral space}
$\maxspec(A)\subseteq \spec(A)$  coincides with the
 topology inherited from ${\rm Spec}(A)$ by restriction.
The standard basis of  closed sets
of $\maxspec(A)$ is given by all
sets of the form  $\mathsf{B}_a=
\{\mathfrak m\in \maxspec(A)\mid a\in\mathfrak m\}$,
letting $a$ range over elements of $A$.
 
\smallskip
An MV-algebra  $A$ is   {\it semisimple} if the intersection
of its maximal ideals is $\{0\}$.
 A maximal ideal $\mathfrak m$ of $A$ is said to have 
 a {\it finite rank}
if the quotient MV-algebra
$A/\mathfrak m$ is finite  (whence
$A/\mathfrak m$ is automatically a finite
MV-chain).

For any compact Hausdorff space $X$ and
MV-algebra  $B$ of continuous functions on $X$,
the maximal ideal $\mathfrak m_x$ of $B$  is  defined by 
 $  \mathfrak m_x=\{f\in B
\mid f(x)=0\}$.

\bigskip
\noindent
{\it Polyhedral topology, \cite{sta}, \cite[\S\S 2-3]{mun11}.}
Let $n=1,2,\ldots $. A set
 $P \subseteq \mathbb R^n$  is
 said to be a {\it polyhedron}  if it is a finite union
 of  closed  simplexes  $S_i\subseteq \mathbb R^n$.
  $P$ need not be convex, nor connected; the
 simplexes    $S_i$   need not have the same dimension.
If  we can choose each $S_i$ 
 with rational vertices then $P$
 is said to be a {\it rational polyhedron}.
 
\smallskip 
  A {\it McNaughton function on}  $\cube$ 
  is a  $\interval$-valued continuous function $f$ 
together with (affine)  linear  polynomials
 $p_1,\ldots,p_u$ with integer coefficients such that
 for each $x\in \cube$ there is $j\in\{1,\ldots,u\}$ satisfying 
 $f(x)=p_j(x).$
 
 \smallskip 
 As a first generalization, for any integers  $n,m>0$ and 
 rational polyhedron
 $P\subseteq \interval^n$,
 a function  $f\colon P\to \interval^m$
 is said to be a   {\it polyhedral \,\,$\mathbb Z$-map}
 if there are McNaughton functions
 $f_1,\dots,f_t$ defined on
$ \interval^n$ 
 such that $f(x)=(f_1(x),\ldots,f_m(x))$  for each  $x\in P$.
% If $f$ is one-one and also $f^{-1}$ is a
% $\mathbb Z$-map, then $f$ is said to be
% a {\it $\mathbb Z$-homeomorphism.} 

\smallskip 
More generally,  given nonempty closed sets
$X\subseteq \interval^{n}$ and $Y\subseteq \interval^m$,
a map $g \colon X\to Y$ is called a {\it $\mathbb{Z}$-map} if
 there exist rational polyhedra  $X\subseteq P\subseteq
 \interval^n$ and $Y\subseteq Q\subseteq \interval^m$,
 and a polyhedral
 $\mathbb{Z}$-map $f \colon P\rightarrow Q$
  such that $g=f\restrict {X}=$ the restriction of $f$ to $X$.
We denote by   $\McN(X)$   the MV-algebra of
  $\mathbb Z$-maps $f\colon X \to \interval$, and 
  say that any such $f$ is a {\it McNaughton function on} $X$.
 
%  The functorial properties of $\mathcal M$ are
%  detailed in  Theorem \ref{theorem:duality}. 
  
    \smallskip
For every rational point  $r\in \mathbb R^n,\,$
$\den(r)$  denotes the least common
denominator of the coordinates of $r$.
We say that $\den(r)$ is the  {\it denominator} 
of $r$.

\begin{theorem}
\label{theorem:residually-finite}
For any finitely generated MV-algebra $A$ the following conditions are
equivalent:
\begin{itemize}
\item[(i)]
A is residually finite.

\smallskip
\item[(ii)]
$A$ is semisimple and its maximal ideals of  finite rank  
 form a dense subset of $\maxspec(A)$.
\end{itemize}
\end{theorem}

\begin{proof} 
For definiteness, let us assume that $A$ has $n$ generators,
$n=1,2,\ldots$.
If  $\psi$ is a homomorphism of $A$ into
a finite algebra $F$,  then 
$F$ is isomorphic to a finite product of finite
MV-chains  $C_1,\ldots,C_k$,\,\,\,
 \cite[Proposition 3.6.5]{cigdotmun}.
   For each $j=1,\ldots,k$ let $\gamma_j$ denote
 the $j$th projection map of $C_1\times\dots\times C_k$
onto $C_j$.  
Suppose  an element $b\in A$ does not belong to
$\ker(\psi).$
Then
$
\gamma_i(\psi(b))\not=0.
$
 for at least one  $i =1,\ldots,k$.
We then have the following well known result:
  {\it If an element $b$ of an MV-algebra $B$
is sent by a homomorphism to a nonzero element
of a finite MV-algebra $F$, then some homomorphism
of $B$ into a finite MV-chain $C$   
sends $b$ to a nonzero element of $C$.}

\smallskip
(i)$\Rightarrow$(ii) Evidently,  $A$ has no infinitesimal
$\epsilon$. For otherwise,
no homomorphism $\chi$ of $A$ into a finite
MV-algebra satisfies  $\chi(\epsilon)\not=0,$
and $A$ is not  residually finite---a contradiction.
%(Equivalently, no homomorphism $\Xi$ of $A$ onto a finite
%MV-chain would satisfy  $\Xi(\epsilon)\not=0,$
%and $A$ would not be residually finite.)
Thus by  \cite[Proposition 3.6.4]{cigdotmun},   $A$ is semisimple. 

We now consider the set of finite
rank ideals of $A$, 
with the intent of proving its denseness
 in the maximal spectral space $\maxspec(A)$.
Using Theorem \ref{theorem:representation}(iii), for 
some nonempty closed subset  $X$ of $\cube$ we first identify    $A$  
with the MV-algebra
 $\McN(X)$ of McNaughton functions on $X$.

Let $\iota\colon x\in X\mapsto \mathfrak m_x\in \maxspec(A)$ 
be the  homeomorphism  
 of $X$ onto $\maxspec(A)$ defined in
  Theorem \ref{theorem:semisimple}(iv). 
  Then  a
 maximal ideal $\mathfrak m\in \maxspec(A)$  has a finite
rank iff the quotient MV-algebra $A/\mathfrak m$ is finite
iff $A/\mathfrak m$ is a finite MV-chain  (because
$A/\mathfrak m$ is simple) iff the   point
$x_\mathfrak m = \iota^{-1}(\mathfrak m)\in X$ given by Theorem
\ref{theorem:semisimple}(vi)  is rational.

Our original task now amounts to proving
that  $X$ has a dense set of rational points.  
  Let us say that a point of $X$ is
{\it irrational} if not all its coordinates are rational.
Arguing by way of contradiction,  let  
$v\in X$ be an irrational point  together with
 an {\it open} rational $n$-simplex  $\mathcal N\ni v$
such that no rational point lies in 
  $\mathcal N\cap X$.
  By \cite[Corollary 2.10]{mun11}
some  McNaughton function $h\colon\cube\to [0,1]$
 vanishes precisely on the (closed)  rational polyhedron
  $\cube\setminus \mathcal N.$ 
Without loss of generality we may assume that the value
$h(v)$ is an irrational number.  
 The restriction $k=h\restrict X$ of $h$
to $X$ belongs to $A=\McN(X)$.
 
Let $\rho$ be a homomorphism  
of $A$ onto a  finite MV-chain.
By Theorem \ref{theorem:semisimple},
for some rational point  $z \in X$\,\,  the quotient
map  $f\in A\mapsto f/\mathfrak m_z$
coincides with $\rho$.
By definition of $\mathcal N$,  the point $z$ lies  in
$X\setminus \mathcal N$, whence  
$\rho(k)=k/\mathfrak m_z=k(z)=0.$
Thus for 
no  homomorphism  $\rho$ of $A$ onto a 
finite MV-chain 
we can have 
$\rho(k)\not=0$.  By the remark 
preceding the proof of (i)$\Rightarrow$(ii),  for  
 no  homomorphism  $\psi$ of $A$ onto a 
finite MV-algebra 
we can have 
$\psi(k)\not=0$.
Since $0\not= k(v)=h(v)\in\interval\setminus \mathbb Q,$  
this contradicts 
the  assumption that
$A$ is residually finite.
Thus $X$ has a dense set of rational points, and
the finite rank maximal ideals of $A$ are
dense in $\maxspec(A)$, as desired.

\smallskip
(ii)$\Rightarrow$ (i)
By Theorem \ref{theorem:representation}(iii) we can 
identify $A$ with $\McN(X)$  for some nonempty 
closed 
space $X\subseteq \cube$  homeomorphic to the
maximal spectral space  $\maxspec(A)$. 
For any nonzero $a\in A$    we will exhibit
  a homomorphism  $\sigma$ of $A$ into a finite MV-algebra,
  such that   $\sigma(a)\not=0$.   
  Let  $y\in X$ satisfy 
 $a(y)>0$. By definition,   
 $a$ is the restriction to $X$ of some McNaughton
 function  $f \colon \cube \to [0,1]$.
Since $f$ is continuous, for some
 open neighborhood $\mathcal R$ of $y$ in $\cube$,
$f$ never vanishes over $\mathcal R$.
With the notational stipulations following 
Theorem \ref{theorem:semisimple}, 
 the assumed 
density   in $\maxspec(A)$  of 
the set of finite rank maximal ideals
yields a rational point $r\in \mathcal R\cap X$ such that  
$f(r)=a(r)>0.$
Its  corresponding maximal
ideal  $\mathfrak m_r$ 
determines  the homomorphism  
$$\sigma\colon l\in A=\McN(X)\mapsto l/\mathfrak m_r
= l(r)\in \interval.$$
The range of $\sigma$
 is the set $V=\{l(r)\mid l\in \McN(X)\}$.
 Let $d$ be the
least common denominator of the coordinates
of $r$.  Since the linear pieces of every $l\in \McN(X)$
are (affine) linear  polynomials with integer
coefficients,  $V$ 
is an MV-subalgebra of the finite MV-chain   $C =\{0,\,1/d,\,\, 2/d,
\ldots, \,(d-1)/d,\,1\}$. (Actually, by McNaughton theorem,
\cite[9.1.5]{cigdotmun},  $V=C$.)
Thus $\sigma(a)=a(r)$ is a nonzero member of the finite
MV-algebra $V=\range(\sigma)$. 
In conclusion,  $A$ is residually finite.
\end{proof}

From 
Theorems \ref{theorem:evans}-\ref{theorem:residually-finite}
we  have:
 
\begin{corollary}
\label{corollary:epi}
Every finitely generated  semisimple  MV-algebra $A$
with a dense set of finite rank maximal ideals
%has the following  property:
%any {\em epimorphism} of $A$ into $A$ is an automorphism.
%A fortiori, $A$ 
 is hopfian.
\end{corollary}

%%%%%%%%%%%%%%%%%%%%%%%%%%
\section{Hopfian MV-algebras and unital $\ell$-groups}
\label{section:examples}

\medskip
\noindent{\it The spectral topology of a unital $\ell$-group.}
An  {\it ideal} of a unital $\ell$-group is
  the kernel of a  homomorphism
  of $(G,u)$.  (In \cite{bkw} ideals are called ``$\ell$-ideals''.)
  We  let $\spec(G,u)$ denote the set of {\it prime}  ideals of 
$(G,u)$, those ideals  $\mathfrak p$ such that
the quotient $(G,u)/\mathfrak p$ is totally ordered.
Unless otherwise specified, all ideals of $(G,u)$ in this paper will
be {\it proper}, i.e., different from $G$.
$\spec(G,u)$ comes   equipped
with the {\it (hull kernel) spectral topology},
whose   closed sets have the form
$\mathsf{F}_\mathfrak j = \{\mathfrak p\in \spec(A)\mid
\mathfrak p\supseteq \mathfrak j\}$,
letting  $\mathfrak j$ range over all 
ideals of $(G,u)$, plus the improper ideal $G$.
 
\smallskip
We  let $\maxspec(G,u)$ denote the set of  maximal
ideals of $(G,u)$  equipped with the topology
inherited  from
$\spec(G,u)$  by restriction.
A basis of closed sets for the {\it maximal
spectral  space} $\maxspec(G,u)$ is given by  the family of sets 
$\,\,\,\mathsf B_{a}
= \{\mathfrak m \in \maxspec(G,u) \mid a\in \mathfrak m \}$  
letting $a$ range over 
elements  of $G$. Owing to the existence of a unit
in $G$, \,\,\,
$\maxspec(G,u)$ is a nonempty compact Hausdorff space,
\cite[10.2.2]{bkw}.
From Theorem \ref{theorem:gamma}
one easily obtains a homeomorphism of 
$\maxspec(G,u)$ onto $\maxspec(\Gamma(G,u)).$

\smallskip
A maximal ideal $\mathfrak m$
of $(G,u)$ is said to have a {\it finite rank}
if the quotient $G/\mathfrak m$ is isomorphic to
the additive group $\mathbb Z$ with its natural order and the unit
coinciding with some integer $v>0$.
Equivalently,  $\mathfrak m\cap [0,1]$
is a finite rank maximal ideal of the MV-algebra
$\Gamma(G,u)$,  (see Theorem \ref{theorem:gamma}).

\smallskip

We say that $(G,u)$ is {\it semisimple}  if the intersection of its maximal
ideals
 is $\{0\}$. Equivalently,
the MV-algebra $\Gamma(G,u)$ is semisimple.
As is well known, $(G,u)$ is semisimple iff it is archimedean
iff it is  isomorphic to
a lattice ordered abelian group of real valued
functions over the compact Hausdorff space 
of its maximal ideals,
with the constant function 1 as the unit, 
 \cite[Corollaire 13.2.4]{bkw}.

 \medskip
 Evidently, there is no residually finite
unital $\ell$-group $(G,u)$ (except the trivial one,
where  $0=u$).
So Theorem \ref{theorem:evans}  has
no direct applicability to  $(G,u)$. 
Hovever, combining
Theorem \ref{theorem:gamma} 
with 
Theorems \ref{theorem:evans}-\ref{theorem:residually-finite}
we  have:

\begin{corollary}
\label{corollary:epi-bis}
Let $(G,u)$ be a finitely generated  semisimple  
unital $\ell$-group $(G,u)$ whose maximal ideals
of finite rank are dense in the maximal spectral space 
$\maxspec(G,u)$.
%Then 
%every epimorphism of $(G,u)$ into $(G,u)$ is an 
% automorphism.
%A fortiori, 
$(G,u)$ is hopfian.
\end{corollary}

An
MV-algebra $F$ is  said to be 
{\it projective}
if whenever $\psi\colon A\to B$ is
a surjective homomorphism and $\phi\colon F\to B$ is a homomorphism,
there is a homomorphism $\theta\colon F\to A$ such that 
$\phi= \psi \circ \theta$.

\begin{corollary}
\label{corollary:hopfians}
Any
finitely presented, any finitely generated projective,
and in particular any finitely generated free MV-algebra
is hopfian.   
\end{corollary}

\begin{proof} By
\cite[Theorem 6.3]{mun11} we may identify the finitely presented
MV-algebra  $A$ with $\McN(P)$  for some rational
polyhedron $P\subseteq \cube.$  It follows that 
$A$ is semisimple and the rational  points
are dense in $P$.
By Theorem \ref{theorem:semisimple},
the finite rank maximal ideals of $A$ are dense in $\maxspec(A).$
By {Theorem  \ref{theorem:residually-finite}},
 $A$ is residually finite.
Since $A$ is finitely generated, then by 
Corollary  \ref{corollary:epi},\,
$A$ is hopfian. 
By   \cite[Proposition 17.5]{mun11}, finitely generated projective MV-algebras are finitely presented. 
Finitely generated free MV-algebras are particular cases of
finitely  presented MV-algebras.
\end{proof}

A unital $\ell$-group $(G,u)$
 is  said to be 
{\it projective}
if whenever $\psi\colon (K,w) \to (H,v)$ is
a surjective homomorphism and 
$\phi\colon (G,u) \to (H,v)$ is a homomorphism,
there is a homomorphism
$\theta\colon (G,u) \to  (K,w)$ such that 
$\phi= \psi \circ \theta$.

\begin{corollary}
\label{corollary:finpres-group}
Any finitely presented, as well as any finitely generated projective
 unital $\ell$-group $(G,u)$
is hopfian.  
\end{corollary}

\begin{proof}
We first remark that  ``finitely presented
unital $\ell$-groups'' as defined
categorically following   Gabriel--Ulmer  
(\cite[19, \S\S II.2,   IX.1]{maclane})
coincide with 
the $\Gamma$ correspondents
of  finitely presented MV-algebras,
(\cite[Theorem 2.2]{cab-forum}, \cite[Remark 5.10(a)]{carrus}).
%A {\it free} unital $\ell$-group is defined as the
%$\Gamma$ correspondent of a free MV-algebra.
%The freeness properties of the
%$\Gamma$ correspondents of free
%MV-algebras are shown in \cite[Proposition 2]{mun-dcds}.
By Theorem \ref{theorem:gamma},
 finitely generated projective unital $\ell$-groups
correspond via $\Gamma$ to
 finitely generated projective MV-algebras.
 
 The desired result then follows from 
{Corollary \ref{corollary:hopfians}}, in view of the preservation
properties of $\Gamma$,  (Theorem \ref{theorem:gamma}).
\end{proof}

\medskip
\subsubsection*{The free $n$-generator unital $\ell$-groups
 $(M_n,1), \,\,\,n=1,2,\ldots$}
\label{subsection:freeg}
For later use in this paper
we introduce here the correspondents via $\Gamma$ 
of finitely generated free MV-algebras.
To this purpose we denote by $(M_n,1)$
      the unital $\ell$-group of all
piecewise   linear continuous functions
$l \colon [0,1]^{n}\to \mathbb R$
such that each linear piece of $l$ has integer
coefficients: the number of pieces of
$l$ is always finite;  the adjective ``linear'' is 
 understood in the affine sense;
$(M_n,1)$ is 
equipped with the
distinguished unit given by the constant function 1
over $ [0,1]^{n}$. 
As an alternative equivalent definition of
$(M_n,1)$,
\begin{equation}
\label{equation:freeg}
\Gamma(M_n,1)=\McNn.
\end{equation}
%More generally,
%for any nonempty closed subset $X\subseteq [0,1]^{n}$ we let
%$\McN_\mathbb R(X)$ denote
%unital $\ell$-group of restrictions to $X$
%of the functions in $\McN_\mathbb R([0,1]^{n})$,
%with the constant function $1\restrict X$ as
%distinguished unit.
%
%
%
The coordinate functions $\pi_{i}\colon [0,1]^{n} \to
\mathbb R,\,\,\,\,\,(i=1,\ldots,n)$
are said to form a {\it free generating set}
of  $(M_n,1)$ because of the
following result:

\begin{proposition}
         \label{proposition:free}
{\rm \cite[4.16]{mun-jfa}} 
$(M_n,1)$   is
generated by the elements $\pi_{1},\ldots,\pi_{n}$
together with the unit $1$.
For every  unital $\ell$-group $(H,v)$,  $n$-tuple
$(v_{1},\ldots,v_{n})$ of elements in the unit interval $ [0,v]$
of $(H,v)$,
and map \,\,$\eta\colon \pi_{i}
\mapsto v_{i}$,   \,\,$\eta$ can be uniquely extended to a
 homomorphism of
$(M_n,1)$ into $(H,v)$.
\end{proposition}

\begin{corollary}
\label{corollary:free-group}
For each  $n=1,2,\ldots$,
$(M_n,1)$ is hopfian.
\end{corollary}

\begin{proof} This follows from
\eqref{equation:freeg}, by an application
of Corollary
\ref{corollary:hopfians} and 
Theorem \ref{theorem:gamma}.
\end{proof}

%%%%%%%%%%%%%%%%%%%%%%%%%%%%%%%%
 \section{Other classes of hopfian  MV-algebras and
 unital $\ell$-groups}
 \label{section:other}
 %%%%%%%%%%%%%%%%%%%%%%%%%%%%%%%
  \subsubsection*{Germinal ideals and quotients}
  Following \cite[Definition 4.7]{mun11}, 
for any  $n=1,2,\ldots$ and $x\in \cube$
the ideal  $\mathfrak o_x$ is defined by 
 $$ \mathfrak o_x=\{f\in
 \McN(\interval^n)\mid f   \mbox{ identically vanishes over some
 open neighborhood of } x  \}.  $$
$ \mathfrak o_x$ is called
 the {\it germinal ideal
 of $ \McN(\interval^n)$ at $x$}.
 Accordingly, the  quotient  MV-algebra
 $\McN(\interval^n)/\mathfrak o_x$ is called the
{\it germinal MV-algebra at $x$},  and for
  each $f\in \McN(\interval^{n})$, 
 $\,\,f/\mathfrak o_x$ is called the
 {\it germ of $f$ (at $x$),}  and is denoted $\Check f$
 whenever $x$ is clear from the context.
 
 If $x$ happens to lie on the boundary of the $n$-cube, 
the neighborhoods of $x$  are understood relative to the
  restriction topology.

\smallskip

%For each ideal  $\mathfrak j$  of $A$
%(including $A$ itself)  let us write  
% $
% \mathsf{F}_\mathfrak j=\{\mathfrak I \in \spec(A) \mid \mathfrak I
% \supseteq \mathfrak j \}.
% $
% Thus in particular,  $F_A=\emptyset$  and
% $F_{\{0\}}=\spec(A).$   One easily checks that
% $F_\mathfrak i\cup F_\mathfrak j = 
% F_{\mathfrak i \cap \mathfrak j}.$
% For any family $\mathfrak j_i$ of ideals of $A$,
% letting  $\bigvee \mathfrak j_i$ denote the smallest
% ideal of $A$ containing all $\mathfrak j_i,$
% it follows that
% $
% \bigcap  F_{\mathfrak j_i} = F_{\bigvee \mathfrak j_i}.
% $
 %
 %
%%%%%4.14 
% \begin{definition}
%\label{definition:spectral-topology}  

\medskip
In general,  an MV-algebra   in either class (i)-(iii) below
need not fall under the hypotheses of Theorem \ref{theorem:evans}.
Thus the following theorem is not a special case of that
  result.

 \begin{theorem}
 \label{theorem:germ}
 The following  MV-algebras are hopfian:
 \begin{enumerate}
 \item[(i)] Simple MV-algebras. 
 
 \smallskip
\item[(ii)]  Finitely generated  MV-algebras with only 
 finitely prime ideals.
%notably   Chang algebra {\rm C}
%$=\{0,\epsilon, 2\epsilon,\ldots,1-2\epsilon, 1-\epsilon,1\}$.  

 \smallskip 
\item[(iii)] For any  rational  point 
 $w$ lying in the interior of the cube  $\cube$, the germinal
 MV-algebra $A=\McNn/\mathfrak o_w$.
 \end{enumerate}
 \end{theorem}
 
 \begin{proof}
(i) By Theorem \ref{theorem:semisimple}(i)-(ii),
  the only endomorphism
of a simple MV-algebra is identity.
 
 \smallskip
(ii)  Let  $\sigma$ be a surjective homomorphism
 of $A$ onto itself, with the intent to prove that $\sigma$
 is injective.
 Let the unital $\ell$-group $(G,u)$
 be defined by $\Gamma(G,u)=A.$
 Then for a unique surjective  homomorphism
 $\tau\colon (G,u)\to (G,u)$ we have  $\sigma=\Gamma(\tau)$.
  Let  $G_{\rm group}$ denote the group reduct of  $(G,u)$.
 The elementary theory of $\ell$-groups
 and their lexicographic products 
\cite[Theorem 3.2]{andfei}  shows that 
$G_{\rm group}$  is a finite product of
 cyclic groups $\mathbb Z.$
 It follows that $G_{\rm group}$ 
 is  hopfian (\cite[\S 3.3 20(d) p. 152]{magkarsol}).
Since, in particular,   $\tau$ is a  homomorphism
of $G_{\rm group}$  onto itself,  $\tau$ is injective.
Going back to $(G,u)$ and $A$,
 also the homomorphism  $\sigma=\tau\restrict[0,1]$
is injective, and $A$ is hopfian.

\medskip
(iii) Let $S^{n-1}$ be the set of unit vectors in $\mathbb R^n$.
 Let $f,g\in \McNn$. By 
  \cite[Proposition 4.8]{mun11},    
\begin{equation}
\label{equation:condition-for-germinal}
f/\mathfrak o_w
=g/\mathfrak o_w  \mbox{  iff  } 
f(w)=g(w)   \mbox{  and  } 
\partial f(w)/\partial e = \partial g(w)/\partial e
\mbox{ for  all   $e\in S^{n-1}$.} 
\end{equation}
Thus we may unambiguously speak of {\it the
directional derivative of a germ}  $\Check f\in A$  and write
$$
\partial \Check f(w)/\partial d = \partial g(w)/\partial d,
\,\,\,\mbox{  $d \in \Rn$,}
$$
independently of the actual representative $g\in \McNn$
of   $\Check f.$
Similarly, we may unambiguously define  $\Check f(w)=g(w)$
and say that {\it  the germ $\Check f$ has value}  $g(w).$
 
 The germinal MV-algebra $A$ has exactly one maximal ideal, namely
 the set $\mathfrak g_w$ given by 
$$
\mathfrak g_w=\{\Check g\in A\mid \mbox{$g\in \McNn$ 
vanishes only at $w$ 
in an open neighborhood of $w$}\}.
$$
 For every prime ideal $\mathfrak p$ of $A$
 the only maximal ideal of $A$ containing $\mathfrak p$  is
 $\mathfrak g_w$.   The   ideals of $A$ containing
 $\mathfrak p$ are all prime, and  form a finite chain (under inclusion),
 whose greatest element is    $\mathfrak g_w$. 
We say that  a prime ideal $\mathfrak q$
of $A$  belongs to
$sub\maxspec(A)$ if  ($\mathfrak q\not= \mathfrak g_w$ and) 
the only ideal properly containing
$\mathfrak q$ is $\mathfrak g_w$.

%  the only maximal ideal $\mathfrak m_\mathfrak p$ of $A$ 
% above $\mathfrak p.$  Since  $A$ has exactly one maximal ideal, for 
%  all $\mathfrak p, \mathfrak q\in
% sub\maxspec(A)$ we have
% $$ \mathfrak m_\mathfrak p =
% \mathfrak m_\mathfrak q =
%  \mbox{the set of germs in $A/\mathfrak o_w$   vanishing at $w$}.
% $$

 \medskip
 \noindent
 {\it Claim.}   Equipped with the  topology
 inherited from  $\spec(A)$ by restriction,
 $sub\maxspec(A)$  is homeomorphic
 to the $(n-1)$-sphere,    
  \begin{equation}
\label{equation:sphere}
 sub\maxspec(A) \cong S^{n-1}.
 \end{equation}

%%%%%%%true true true by rationality of w
 As a matter of fact,  let  $\mathfrak p\in sub\maxspec(A)$
 and  $\Check f\in \mathfrak p$.  
%By  \eqref{equation:condition-for-germinal} we may assume that
%the zeroset of $f$ is the singleton point $\{w\}$. 
It is not hard to see that 
the directional derivative  
$\partial \Check f(w)/\partial d $
of $\Check f$ at $w$ must vanish along some
direction $d\in S^{n-1}$. 
For otherwise,  since the map 
$d\in S^{n-1}\mapsto \partial f(w)/\partial d\in \mathbb R$ 
is continuous and $S^{n-1}$ is compact, there is
$\rho>0$ such that $\partial f(w)/\partial d\geq \rho$ 
for all $d\in S^{n-1}.$
Thus every function  $g\in \McNn$ 
 vanishing
at $w$ is dominated by a suitable integer multiple of $f$.
It follows that  $\mathfrak p$
coincides with the maximal ideal $\mathfrak g_w$ of $A$,
 against our hypotheses
that  $\mathfrak p\subsetneqq \mathfrak g_w$. 
Another application of the
continuity of the map  $d \mapsto \partial \Check f(w)/\partial d$
 shows that
the set  $S_{\Check f}$ of unit vectors  $v$
such that $\partial \Check f(w)/\partial v=0$ is a {\it closed}
nonempty subset 
of $S^{n-1}$.
A compactness argument now shows that
the intersection of the sets  $S_{\Check g}$, where
$\Check g$ ranges over all elements of $\mathfrak p$,
is a closed nonempty subset $D_\mathfrak p$ of  $S^{n-1}$.
Since the point $w$ is rational and
  $\mathfrak p$ is prime, a routine  separation
  argument  shows that
   this set must
be a singleton, $D_\mathfrak p=\{d_\mathfrak p\}$ with
$d_\mathfrak p\in  S^{n-1}$.

 Conversely, 
  every    unit vector $d\in S^{n-1}$ determines
 the prime ideal   $\mathfrak p_d$  of $A$ 
 consisting of the germs of $A$ vanishing at $w$ and whose
 directional derivative at $w$ along $d$ vanishes.
 Since $w$ is a rational point, $\mathfrak p_d$ does not
 coincide with $\mathfrak g_w$. 
 The first part of the proof of the claim can be used to show that
 there is no prime ideal $\mathfrak p$ between $\mathfrak g_w$ and
  $\mathfrak p_d$. So  $\mathfrak p_d$ belongs to 
 $sub\maxspec(A)$.

In conclusion, the  maps 
\begin{equation}
\label{equation:primes-directions}
 \mathfrak p\in sub\maxspec(A)\mapsto
 d_\mathfrak p\in S^{n-1} \mbox{ and } \,\, d \in S^{n-1}\mapsto
  \mathfrak p_d\in sub\maxspec(A)
\end{equation}
are inverses of each other.  By  definition of the
prime spectral topology, these maps are homeomorphisms.
Our claim is settled.

\smallskip
For  every $ g \in \McNn$
we can now expand  \eqref{equation:condition-for-germinal}
as follows:
\begin{equation}
\label{equation:stronger}
g\in \mathfrak o_w
 \mbox{  iff  } 
\Check g(w) =0 
\mbox{ and }\partial\Check g(w)/\partial d
\mbox{ for all  $d\in \Rn$} 
 \mbox{  iff  } 
\Check g \in \mathfrak p
 \mbox{ for all $\mathfrak p\in sub\maxspec(A)$}. 
\end{equation}
%Incidentally, we have obtained the well known identity 
% $\mathfrak o_w=\bigcap \{\mathfrak p\in sub\maxspec(A)\}. $
%%% NEEDS TO PASS TO GERMS

\medskip
To conclude the proof, by
 way of contradiction  let  
$\sigma\colon A\to A$ be a surjective endomorphism
which is not  injective, i.e., there exists  
$h \in\McNn$  satisfying
$ 
0\not= \Check h   \in \ker(\sigma).
$
Since, trivially,  
$\Check h(w)=0$, then by \eqref{equation:condition-for-germinal}  we must have
$
\partial \check h(w)/\partial r >0 \mbox{ for some } r\in S^{n-1}.
$
Thus $\check h$ does not belong to the prime ideal 
$\mathfrak p_r$
of \eqref{equation:primes-directions}.
It follows from \eqref{equation:stronger} that 
\begin{equation}
\label{equation:crucial}
\ker(\sigma)\nsubseteq \mathfrak   p_r \in 
sub\maxspec(A). 
\end{equation}

%%%FACOLTATIVO
%Since for fixed $a\in \cube$ and $f\in \McNn$, the
% directional derivative  
% $\partial f(a)/\partial d$   varies
%continuously with the direction $d\in \Rn$,   there is
%a nonempty relatively open set 
%$O_h\subseteq S^{n-1}$ of
%unit vectors in the unit sphere $S^{n-1}\subseteq \Rn$ centered at 0
%such that 
%$
%\partial \Check h(w)/\partial s >0 \mbox{ for all  } s\in O_h.
%$
%Moreover, $O_h\not=S^{n-1}$, 
%because $\Check h$ is not the zero germ.

We now inspect the prime spectrum of the quotient MV-algebra
$A/\ker(\sigma)$. 
The elementary ideal theory of MV-algebras,
 \cite[pp. 16-18]{cigdotmun},\,\,\, yields  a
 one-one  inclusion preserving correspondence
 between $\spec(A/\ker(\sigma))$
and  the set   of prime ideals of  $A$ containing  
$\ker(\sigma).$
Since this correspondence is 
 also a homeomorphism, 
the set  $sub\maxspec(A/\ker(\sigma))$
equipped with the restriction topology is homeomorphic
to  the set $\mathfrak P$ of all 
$\mathfrak p\in sub\maxspec(A)$ containing  $ \ker(\sigma)$.
Now \eqref{equation:sphere} and \eqref{equation:crucial} show that
$\mathfrak P$  is homeomorphic to a {\it proper\/} subset of
$ S^{n-1}$. 
As is well known, 
$ S^{n-1}$ is not homeomorphic to any of its proper subsets
 (see, e.g., \cite[7.2(6), p. 180]{ams} 
 for a proof).  As a consequence, 
 $$sub\maxspec(A/\ker(\sigma))\not\cong sub\maxspec(A).$$
By contrast, from the isomorphism
$
A=\sigma(A) \cong A/\ker(\sigma)
$
we obtain  
$
\spec(A)\cong \spec(A/\ker(\sigma)),
\,\,\ \maxspec(A)\cong\maxspec(A/\ker(\sigma)).
$
Since the homeomorphism of
$\maxspec(A)$  onto $\maxspec(A/\ker(\sigma))$ is order preserving
(and so is its inverse) we finally obtain 
$$ 
sub\maxspec(A)\cong sub\maxspec(A/\ker(\sigma)),
$$
a contradiction showing that $A$ is hopfian.
\end{proof}

 \begin{corollary}
 	\label{corollary:germ-group}
 The following classes of unital $\ell$-groups are hopfian:
 \begin{enumerate}
 \item[(i)] Simple unital $\ell$-groups. 
 
 \smallskip
 \item[(ii)]  Finitely generated  unital $\ell$-groups with
 only finitely many prime ideals.  
 
  \smallskip
 \item[(iii)] For $w$ a rational point
  lying in the interior of the cube  $\cube$, the 
  {\em germinal
unital $\ell$-group $(G,u)$   at $w$},  defined by
$\Gamma(G,u)=\McNn/\mathfrak o_w$.
(See \cite[\S 10.5]{bkw} for an equivalent definition independent
of MV-algebras). 
 \end{enumerate}
 \end{corollary}
 \begin{proof} By   Theorems
\ref{theorem:germ} and \ref{theorem:gamma}.
 \end{proof}

  \begin{theorem}
 \label{theorem:manifold}
  Let $A$ be a semisimple
 MV-algebra.

\smallskip 
 (i) If $\maxspec(A)$ is an
 $n$-dimensional manifold without boundary then  $A$ is hopfian.

  \smallskip
 (ii)  If
 $\maxspec(A)$ is 
 an $n$-manifold 
 with boundary, and
$A$ has a generating set with $n$
 elements, then  $A$ is hopfian.

  \smallskip
 
 (iii) More generally,  $A$ is hopfian if
  $A$ has a generating set of $n$ elements
  and   $\maxspec(A)$ is homeomorphically
  embeddable onto a subset $X$ of $\cube$
  coinciding with the closure of its interior.
 \end{theorem}

 \begin{proof}
 (i)
  Let $\kappa$ be the cardinality of a generating set of 
 $A$.  By Theorems
 \ref{theorem:representation}(iii)
 and \ref{theorem:semisimple},
 we may identify the semisimple MV-algebra
$A$ with
the MV-algebra  $\McN(N)$ for some
homeomorphic copy  $N\subseteq \interval^\kappa$
of the maximal spectral space  $\maxspec(A).$
$N$ is an $n$-manifold with boundary embedded into
$\interval^\kappa.$
 
 By way of contradiction,  suppose  $\sigma$
is  a  homomorphism  of $\McN(N)$
 onto $\McN(N)$  and $\sigma$ is not one-one.
 By Lemma \ref{1.2.8} 
we have an   isomorphism
\begin{equation}
\label{equation:primoiso}
  \McN(N)/\ker{\sigma}\cong
\McN(N). 
%= \sigma(\McN(N))
%\mbox{ with } \tau(x/\ker(\sigma))=\sigma(x)
% \mbox{ for all } x\in \McN(N).  
\end{equation}
Let  $N' =\bigcap \{f^{-1}(0)
 \mid f\in \ker(\sigma)\}$.
 By Theorem \ref{theorem:semisimple},
 $N'\cong \maxspec(\McN(N'))$.  
Since  $\ker(\sigma)
\not=\{0\}$, 
  $N'$ is a  proper  nonempty closed subset of $N$,
(\cite[Proposition 3.4.2]{cigdotmun} and definition of
maximal spectral topology).
Since   $\sigma(\McN(N))=\McN(N)$
 is semisimple,
  $\ker(\sigma)$
 is an intersection of maximal ideals.
By Proposition \ref{proposition:3.4.5}, 
 $ \McN(N)/\ker{\sigma}\cong \McN(N').$
By \eqref{equation:primoiso},  
$\McN(N)$ and $\McN(N')$ are isomorphic,
whence their respective maximal spectral spaces are
homeomorphic, and so are  
$N'$ and $N$.

  However, since $N$
is a manifold without boundary, every one-one continuous map
of  $N$ into    $N$  is surjective. (See
  \cite[Corollary 6.3]{var} for a proof.) 
    We have thus reached a contradiction,
  showing that $A$ is hopfian.

 \smallskip  
  (ii)  For  some integer $n>0$, 
Theorems 
 \ref{theorem:representation}(iii) and \ref{theorem:semisimple}, 
 enable us to 
identify $\maxspec(A)$ with a closed subset $X$ of $\interval^n$,
and write $A=\McN(X).$   
Our hypothesis about $\maxspec(A)$ entails that
 $X$  is equal to the closure of its interior
 in $\cube$.
Thus the rational points are dense in $X$, whence
(by Theorem \ref{theorem:semisimple}), 
  the finite rank maximal ideals of $A$ are dense in 
  $\maxspec(A)$. By  Theorem
  \ref{theorem:residually-finite},  $A$ is residually finite.
  By hypothesis,  $A$ is finitely generated. By 
  Corollary  \ref{corollary:epi},
  $A$ is hopfian.
 
  \smallskip 
  (iii) Same proof as for (ii).
\end{proof}

\begin{corollary} Let $(G,u)$ be a semisimple
 unital $\ell$-group.
 Suppose the maximal spectral  space of
  $(G,u)$     is a
 manifold without boundary.  Then $(G,u)$ is hopfian.
 \end{corollary}

%%%%%%%%%%%%%%%%%%%%%
\section{Hopfian vs. non-hopfian}
 %%%%%%%%%%%%%%%%%%%%%

In general, the finite rank maximal ideals
of an $n$-generator  semisimple  MV-algebra $A$ are not
dense in $\maxspec(A),$  even when $n=1.$ 
So  the following  
result is not derivable from Theorem \ref{theorem:evans}:

 \begin{proposition}
 \label{proposition:onegen}
Every one-generator semisimple  MV-algebra 
$A$  is hopfian.
 \end{proposition}

\begin{proof}
$A$ is a quotient of the free one-generator
MV-algebra $\McN(\interval)$.
By Theorems
 \ref{theorem:representation}(iii)
 and \ref{theorem:semisimple},
  we may identify the semisimple MV-algebra
$A$ with
the MV-algebra  $\McN(X)$ for some
homeomorphic copy  $X\subseteq \interval$
of the maximal spectral space  $\maxspec(A).$

 By way of contradiction,  suppose
 there is a  homomorphism $\sigma$ of $\McN(X)$
 onto $\McN(X)$  which is not one-one.
 By Lemma \ref{1.2.8} 
we have an   isomorphism
\begin{equation}
\label{equation:primoiso-bis}
\tau\colon \McN(X)/\ker{\sigma}\cong
\McN(X) 
%= \sigma(\McN(X))
\mbox{ with } \tau(x/\ker(\sigma))=\sigma(x)
 \mbox{ for all } x\in \McN(X).  
\end{equation}

%\commento{Proposition 3.4.2
% is the inclusion reversing
%map from the ideals of $A$ 
%into the closed subsets of its maximal space
%}

Let  $X' \subseteq X$ be the intersection of the zerosets
$f^{-1}(0)$  of the functions  $f\in \ker(\sigma)$.
  By Theorem \ref{theorem:semisimple},
 $X'$ is homeomorphic  to $\maxspec(\McN(X'))$. 
Since  $\sigma$ is not one-one, $\ker(\sigma)
\not=\{0\}$.
By definition of maximal spectral
topology,
  $X'$ is a 
{\it proper} nonempty
 closed subset of $X$,
 (see \cite[Proposition 3.4.2]{cigdotmun} for details).

Note that  $\ker(\sigma)$
 coincides with the  intersection of all maximal ideals
 containing it
(because  $\sigma(\McN(X))=\McN(X)=A$,  which is
 assumed to be  semisimple).
By   \cite[Proposition 3.4.5]{cigdotmun}, 
 $$
  \McN(X)/\ker{\sigma}\cong \McN(X').  
 $$
Then by  \eqref{equation:primoiso-bis}
  we have an isomorphism
 $$\eta\colon  \McN(X)\cong  \McN(X').$$
 Correspondingly,
% , from 
% $X\cong\maxspec(\McN(X))$ and 
% $X'\cong\maxspec(\McN(X'))$, 
 $\eta$ yields a homeomorphism
\begin{equation}
\label{equation:homeomorphism}
 \theta\colon X'\cong X
\end{equation}
%
%
%
%\commento{here were the last boldfaced phrases}
%
%
% 
  such that 
 $$
 \McN(X')/\mathfrak m\cong \McN(X)/\theta(\mathfrak m)
 \mbox{ for each $\mathfrak m\in \maxspec( \McN(X'))$}.
 $$
By  Theorem \ref{theorem:semisimple}(i), both 
quotients  
$\McN(X')/\mathfrak m$ and $ \McN(X)/\theta(\mathfrak m)$ 
 are uniquely isomorphic to the
same subalgebra $I_\mathfrak m$ of the standard MV-algebra $\interval$, and we can safely   write 
  $$
 \McN(X')/\mathfrak m = \McN(X)/\theta(\mathfrak m)
 = I_\mathfrak m
 \mbox{ for each }
 \mathfrak m\in \maxspec( \McN(X')).
 $$
 Thus  for every 
$x \in X'$,  $\theta$ preserves the
group  $$G_x= 
\mathbb Zx+\mathbb Z= \McN(X)/\mathfrak m_x $$
 generated 
 by  $x$ and the unit  1
 in the additive group $\mathbb R$. In symbols,
\begin{equation}
\label{equation:group-preserve}
G_x=G_{\theta(x)}.
\end{equation}
In particular, $\theta$ preserves the denominators of
rational points of $X'$, (if any).

\medskip

%%
%%
%\commento{needs elementary properties (a) and (b) 
%as in the proof for manifold without boundary. Moreover, it needs
%(c) that the injecttion preserves groups $G_x$.
% Then the  elementary ideal theory of MV-algebras,
% \cite[pp. 16-18]{cigdotmun},   + theory of semisimples  }
%%
%%
%%
%
%{\bf    Let  $id_X$ be the 
% identity function on $X$.  Let $p=\sigma(id_X)\in \McN(X)$.
% [[[Then $\sigma=-\circ p$. Hence $range(p)\subseteq X$]]] From ontoness of $\sigma$, some $q\in \McN(X)$
%  satisfies $q\circ p=id_X.$
%This
%shows that   $p$ is a
% one-one (continuous) map of $X$ onto a subset $X'$ of $X$.
% $X'$ coincides with the intersection of the zerosets of all functions
% in $\ker(\sigma).$ Since $\sigma$ is not one-one  $X'$
% is a proper subset of  $X$.  $q$ does the converse embedding
% of $X'$ into $X$.
%  [[[Restriction to $X'$ is an isomorphism
% of  $\McN(X)$ onto $\McN(X')$.]]]
%
%Easy is the preservation of rational denominators
%  because if it gets a proper divisor then
%it cannot be inverted.
%% ,   \cite[Remark after Definition 3.4]{cab-forum}.
%Thus  $p$ preserves the denominators of
% all rational points in its domain. 
% 
%$\sigma=-\circ p$ is isomorphism and so is
%$\sigma^{-1}=-\circ q$
%
%
%Moreover, 
%  for every 
%$x\in X$,  $p$ preserves the
%group  $$G_x= 
%\mathbb Zx+\mathbb Z $$
% generated 
% by  $x$ and the unit  1
% in the additive group $\mathbb R$. In symbols,
%$$G_x=G_{p (x)}.$$
%}

 Since  $X'$ is a proper subset of $X$, let
$$c \in X\setminus X'.$$

\medskip
\noindent {\it Case 1.} $c$ is rational. Say
$\den(c)=d$.
There are only finitely many points in
$[0,1]$ of  denominator $d$.
By  the  pigeonhole principle  it is impossible
that $\theta^{-1}$
maps the set $D\subseteq X$ of points
of  denominator $d$ one-one {\it onto}  $D\setminus \{c\}$.

\medskip
\noindent {\it Case 2.}
 $c$ is irrational.
For any two irrational points  $a,b,\in [0,1]$
 a routine verification shows 
 $G_a=G_b$ iff  $a=b$ or $a=1-b$.
 Upon writing $c\not=\theta^{-1}(c)=
\theta^{-1}(\theta^{-1}(c))$
 we obtain a contradiction with the fact that
  the one-one map $\theta$ of \eqref{equation:homeomorphism}
has the preservation property \eqref{equation:group-preserve}.

\smallskip
Having thus proved that both cases
are impossible, we conclude that  $A$ is hopfian.
\end{proof} 

\begin{figure} 
    \begin{center}                                     
    \includegraphics[height=8.0cm]{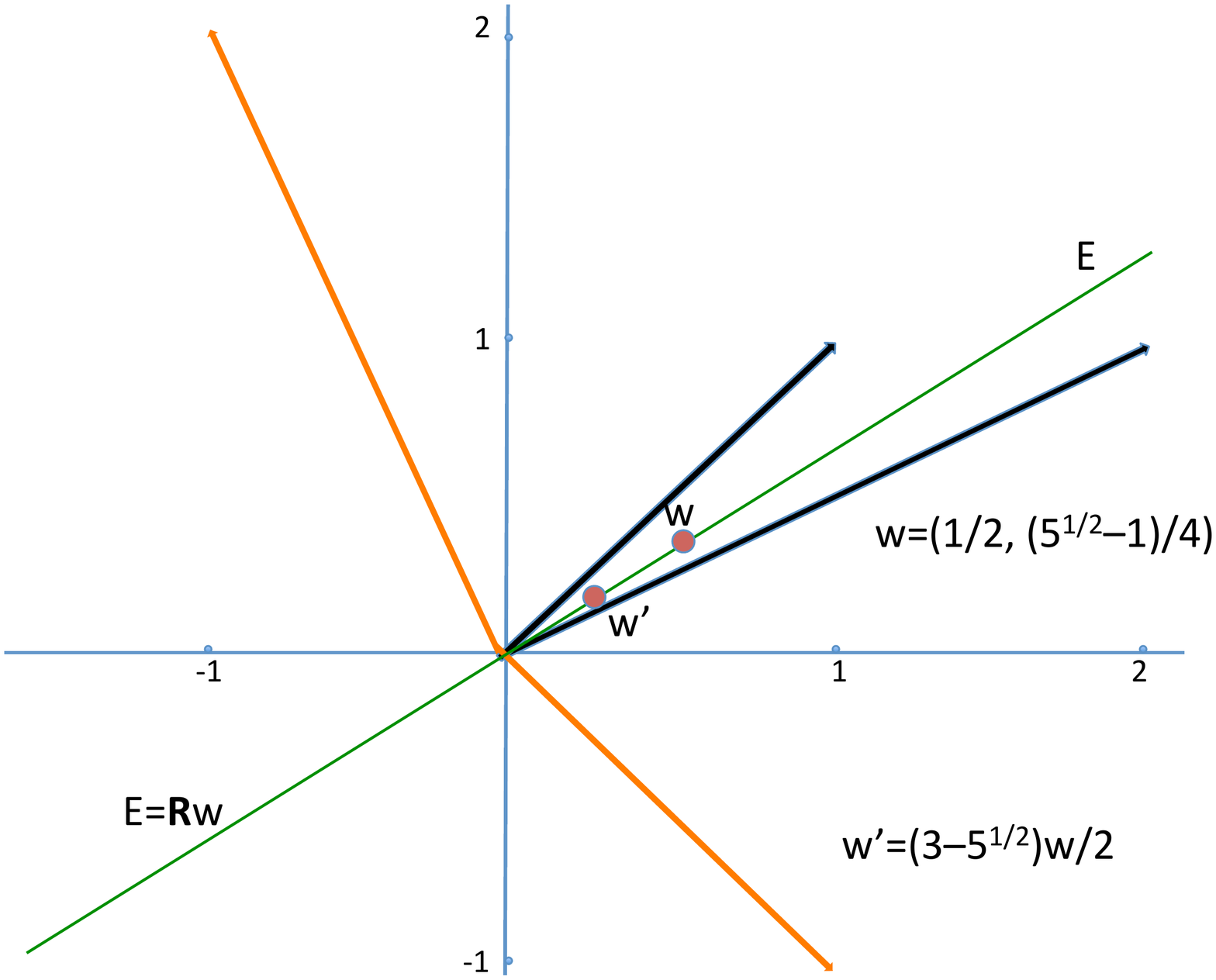}    
    \end{center}                                       
 \caption{\small  
The line $E$ is one of the two eigenspaces
of the  unimodular matrix  with rows (2,1) and (1,1),
and   eigenvalue $1< 1/\lambda \notin \mathbb Q$.
$E$ is also an eigenspace of the inverse matrix 
$L$.  For each $v \in E$,\,
  $Lv =\lambda v= (3-5^{1/2})v/2$.
Let $w=(1/2,(5^{1/2}-1)/{4})	\in E \cap [0,1/2]^2.$ ThenÊ
$L$ yields a  one-one $\mathbb Z$-map  $l^*$  of
 $[0,w] $ onto  $[0,w']=[0,Lw]\subsetneqq [0,w].$ 
  Theorem \ref{theorem:eigen} shows that the MV-algebra
  $\McN([0,w])$ is not hopfian.
}%%caption  
    \label{figure:eigen-tris}                                                 
   \end{figure}

As will be shown by our next result, 
when $n\geq 2$  the analog of
Proposition \ref{proposition:onegen}
fails in general  for $n$-generator semisimple 
MV-algebras.

\smallskip

Throughout the rest of this section   the adjective ``linear''
is  understood in the homogeneous sense.

\begin{theorem} 
\label{theorem:eigen}
Let  $L$ be an
$n \times n$ matrix with integer entries and determinant
equal to 
\,$\pm 1$.
Suppose $L$ has
 a  one-dimensional linear eigenspace $E$
 with eigenvalue $0< \lambda<1$, and 
 $E$ has a nonempty intersection with the
 interior of   $[0,1]^n$.
Then  the semisimple $n$-generator
MV-algebra $\McN(E\cap [0,1/2]^n)$
is not hopfian and is not residually finite.
\end{theorem}

\begin{proof}   $L$ acts on vectors
$x=(x_1,\ldots,x_n)\in \Rn$ as the linear
transformation  $x\mapsto Lx$, also denoted $L$. 
%$$
%\Lambda
%\colon x=(x_1,\ldots,x_n)\in \Rn\mapsto Lx\in \Rn.
%$$
By hypothesis,
 $L$ preserves denominators of
rational points.

Observe that no nonzero rational point
lies in $E$.
For otherwise (absurdum hypothesis) if 
$0\not=r \in E\cap \mathbb Q^n$  
then from $0< \lambda < 1$ we get
an infinite set  $r, \,\,L(r),\,\, L(L(r)),\,\,\ldots$
 of rational points in $\interval^n$
all of denominator $\den(r)$. This is
impossible,   because for each
$d=1,2,\ldots,$  the cube  $\cube$ contains only finitely many 
rational points of denominator $d$. 

Let us write for short  $E'=E\cap [0,1/2]^n.$
In view of  Theorem \ref{theorem:semisimple},
since  $E'$ has no nonzero rational point then a fortiori 
the finite rank maximal ideals
of the MV-algebra  $\McN(E')$ do not form a dense
set  in the maximal spectral space $\maxspec(\McN(E'))$.
By Theorem \ref{theorem:residually-finite}, the semisimple
MV-algebra  $\McN(E')$ is not residually finite.

Since   $0< \lambda <1$, 
there is an open neighborhood $\mathcal N$ of 
$E'$ in $\cube$ such that the restriction   
%$\Lambda'=
$L \restrict \mathcal N$ maps $ \mathcal N$  into $ \cube.$
Without loss of generality,  the closure of $\mathcal N$ 
is a rational polyhedron, whence
by \cite[Proposition 3.2]{mun11},  $L \restrict \mathcal N$
is extendible to a $\mathbb Z$-map 
%$\Lambda^*
$l  \colon \cube\to \cube$.
Let  the map $l^*\colon E'\to E'$
be defined by   
$$l^*=l\restrict E'=L\restrict E'.$$
For any   $f\in \McN(E')$
the composite function $f\circ l^*$ is 
a member of  $ \McN(E')$.
So  let the map $\sigma\colon  \McN(E')
\to  \McN(E')$  be defined by 
$$
\sigma(f)=  f \circ l^*, \mbox{ for all }  f\in  \McN(E'). 
$$ 

\medskip
\noindent
{\it Claim.}  $\sigma$ is a   homomorphism
of  $ \McN(E')$ onto itself.

As a matter of fact, the inverse  $L^{-1}$ determines a
 linear transformation 
 $x\mapsto L^{-1}x$ 
 of the vector space $\Rn$ one-one onto $\Rn$.
% 
%
%
%  Further,
%for all  $i=1,\ldots,n$  
%we have 
%$$
%(V_i\circ \Lambda)(x_1,\ldots,x_n)=x_i \mbox{ for all }
%x=(x_1,\ldots,x_n)\in \mathbb R^n.
%$$
Since  $L^{-1}$ is an integer matrix and
the nonzero extreme of the segment $E'$ lies in the interior of $\cube$,
for each $i=1,\dots,n$
there is a  McNaughton function
$k_i\colon \cube \to \interval$
such that 
$
(k_i\circ l)(x_1,\ldots,x_n)=x_i \mbox{ for all }  x \in E'.
$
The restriction $k_i^*=k_i\restrict E'$ is a McNaughton function
on $E'$ satisfying the identity
$
(k^*_i\circ l^*)(x_1,\ldots,x_n)=x_i \mbox{ for all }  x\in E'.
$
Since each
component of 
the identity $\mathbb Z$-map on $E'$
is in the range of $\sigma$ then, by
  definition of $\sigma$,   every
member of $\McN(E')$ is in the
range of $\sigma$, and our claim is settled.

\smallskip
Now,
$l^*$  is one-one but is not surjective, because it
 shrinks
$E'$ to the closed set  $l^*(E')=\lambda E' \subsetneqq E'$.  
 There is an open rational $n$-simplex $T\subseteq \cube$ such that
$T\cap (E'\setminus \lambda E' )\not=\emptyset$
and $T\cap \lambda E' =\emptyset.$
  By \cite[Corollary 2.10]{mun11}
   there is a McNaughton function  $g \in \McNn$
whose zeroset coincides with the (closed) rational
polyhedron  $\cube\setminus T.$
So $g$ vanishes identically on $\lambda E' $, but
does not vanish identically over  $E'.$
The restriction $g\restrict E'$ is a member of 
$\McN(E')$, and so is the composite function
$(g \restrict E')\circ l^* = \sigma(g)$.
Since $(g \restrict E')\circ l^* $
vanishes identically over $E'$,
  $g \restrict  E' $ is a nonzero member of $\ker(\sigma)$.  
We have just shown that the
surjective homomorphism $\sigma$   is
not injective. 
Thus $\McN(E')$ is not hopfian.
\end{proof}

See  Figure \ref{figure:eigen-tris} for an
example of a semisimple
non-hopfian two-generator MV-algebra.

\begin{corollary}
With the notation of  the foregoing
theorem, let the unital $\ell$-group $(G,u)$  be defined by
$
\Gamma(G,u) = \McN(E\cap [0,1/2]^n).
$
Then $(G,u)$ is finitely generated, semisimple and 
non-hopfian.
\end{corollary}

\medskip

We refer to \cite[\S 7]{mun11} for free products of MV-algebras.
\medskip

 The  {\it Chang algebra\/} {\rm C}
$=\{0,\epsilon, 2\epsilon,\ldots,1-2\epsilon, 1-\epsilon,1\}$ is
defined as $\Gamma(\mathbb Z +_{lex} \mathbb Z, (1,0))$,
where  $+_{lex}$ denotes the lexicographic product,
\cite[p. 141]{bkw}.
% (also see \cite[9.1, p. 102]{mun11}).

\smallskip

\begin{theorem}
\label{theorem:non-hopf}
The following MV-algebras are not hopfian:
\begin{enumerate}
\item[(i)] Free MV-algebras over infinitely many free generators.

\smallskip
\item[(ii)] The free product ${\rm C}\amalg {\rm C}$. 

\smallskip 
\item[(iii)]
All  countable boolean algebras.  
\end{enumerate}
\end{theorem}
\begin{proof}
(i)  
Let $Free_\omega=\McN(\interval^\omega)$ be the 
free   MV-algebra over the free generating set
$\{\pi_1,\pi_2,\ldots\}$,  where  $\pi_i$
is the $i$th coordinate function over the Hilbert cube
$\interval^\omega$.  Let $\sigma$  be the unique endomorphism
of $Free_\omega$ extending the map
$\pi_1 \mapsto \pi_1,\,\, \pi_{n+1}\mapsto \pi_n, \,\,n=1,2,\ldots$.
Then $\sigma$ is a surjective non-injective
homomorphism, showing that $Free_\omega$  is not 
hopfian. The case of $Free_\kappa$ for a  cardinal
$\kappa>\omega$
is similar.

\smallskip
(ii)  For this proof we
 assume familiarity with MV-algebraic free products
(\cite[\S 7]{mun11}) 
 and with the basic properties of the
  free  $\ell$-group  $\mathcal A_2$
over two free generators:
this is the  $\ell$-group  of all continuous piecewise {\it homogeneous}
linear functions  $f\colon \mathbb R^2 \to \mathbb R,$ where each piece
has integer coefficients,  \cite[Theorem 6.3]{andfei}.
Directional derivatives and
values  of  germs are  
understood  and denoted as in the proof of 
Theorem \ref{theorem:germ}(iii).

It is not hard to construct an isomorphism $\gamma$ of
 Chang algebra C onto  the germinal quotient 
$\McN(\interval)/\mathfrak o_0$ of $\McN(\interval)$
at the origin 0. To this purpose, we note that   
two  McNaughton functions  $f,g\in \McN(\interval)$
have the same germ at 0  iff $f(0)=g(0)$  and  
  $\partial f(0)/\partial x^+=\partial g(0)/\partial x^+.$
So any germ $\Check f$ with value 0  
is uniquely determined by the value
$\partial\Check f(0)/\partial x^+\in \{0,1,2,\ldots\}$
of its
  directional derivative along the positive direction of
  the $x$-axis.
  The map $\gamma$  sends  the element $n\epsilon \in$ C to
  the germ  $n\Check{x}$ whose directional derivative is $n$.  
  In particular, the smallest infinitesimal $\epsilon$ of C corresponds to
 the germ $\Check x$ of the identity function $x$
 on $\interval$.
 Symmetrically,  $\gamma$  sends the element  $1-m\epsilon\in$ C
 to the germ  $\neg m\Check x$ with  value 1 and directional derivative equal to $-m.$
 We will henceforth identify C and $\McN(\interval)/\mathfrak o_0.$
 
To construct the free product MV-algebra
  C$\amalg$C, following 
\cite[Theorem 7.1]{mun11},  for every
$f\in \mathfrak o_0$ we first embed $f$ into
 $\McN(\interval^2)$ by
 {\it cylindrification
along the $y$ axis}: in other words, we transform
$f$ into the function $f\circ \pi_1\in \McN(\interval^2)$, with
$\pi_1\colon \interval^2\to \interval$ the first coordinate function.
We similarly transform 
$f$  into the function 
$f\circ \pi_2$, its cylindrification along the $x$ axis.
 %
%\bigskip
%we prepare  two copies
% $\mathfrak o'_0$ and  $\mathfrak o''_0$ of the
%germinal ideal $\mathfrak o_0$. Then we embed
%each germ  $f/\mathfrak o_0$ of  $\mathfrak o'_0$
%into   $\McN(\interval^2)$ by cylindrification
%along $y$.
%Similarly, we embed
%every germ of $\mathfrak o''_0$
%into   $\McN(\interval^2)$
%by cylindrification along   $x$.
 %
A routine verification shows 
 that the ideal of $\McN(\interval^2)$ generated by
the set of all these cylindrified functions
%$\mathfrak o'_0\cup \mathfrak o''_0$ 
coincides with the germinal ideal $\mathfrak o_{(0,0)}$
of $\McN(\interval^2)$ at the origin $(0,0)\in \interval ^2$.
Every germ  $\Check g\in \McN(\interval^2)/\mathfrak o_{(0,0)}$
can only have the two possible values  0 and 1.
We have an isomorphism   
\begin{equation}
\label{equation:amalg}
\mbox{C$\amalg$C}
\cong \McN(\interval^2)/\mathfrak o_{(0,0)}.
\end{equation}

Let $Q$ denote the first  quadrant in $\mathbb R^2.$
A germ   $\Check f\in \McN(\interval^2)/\mathfrak o_{(0,0)}$ 
having value 0 is completely characterized
by the unique continuous piecewise {\it homogeneous} linear function  
$\Vec f\colon Q\to \mathbb R_{\geq 0}$
whose directional derivatives at $(0,0)$
along all (unit) vectors $d\in Q$ coincide with those
of  $f$. Each linear
piece of $\Vec f$ has integer coefficients.  If
$\Check g=\Check f$ then $\Vec g=\Vec f$.
Conversely, for any continuous 
piecewise homogeneous linear function  
$l \colon Q\to \mathbb R_{\geq 0}$ with integer coefficients
there is precisely one
germ $\Check h \in \McN(\interval^2)/\mathfrak o_{(0,0)}$
%\,\,(with $h\in \McN(\interval^2)$)\,\, 
such that $l =\Vec h$.
%  (for some $h\in \McN(\interval^2)$). 
This one-one correspondence provides a convenient 
identification of  
the set of germs in 
$\McN(\interval^2)/\mathfrak o_{(0,0)}$  having value 0 at the origin,
with  the
set $\mathsf P $ of continuous   
piecewise homogeneous linear functions $p \geq 0$ 
with integer coefficients, defined
on $Q$. Elements of
$\mathsf P$ are acted upon by the 
pointwise addition, truncated substraction,
 and max min operations of $\mathbb R.$
Similarly, each germ in
$\McN(\interval^2)/\mathfrak o_{(0,0)}$ 
 with value 1 can be
identified with  $1-p$ for a uniquely determined  $p\in\mathsf P$. 
Accordingly, we may write 
\begin{equation}
\label{equation:visuale}
\McN(\interval^2)/\mathfrak o_{(0,0)}
={\mathsf P} \cup (1-{\mathsf P}).
\end{equation}

Let $H$ be the
 $\ell$-group of continuous 
 piecewise homogeneous linear   functions
$f\colon Q \to \mathbb R$  with
integer coefficients. Recalling the definition of
the free $\ell$-group  $\mathcal A_2$
we can write 
\begin{equation}
\label{equation:acca}
H=\mathcal A_2\restrict Q =\{f\restrict Q\mid f\in \mathcal A_2\}
\,\,\,\mbox{   and   }\,\,\, \mathsf{P}=H^+=\{p\in H\mid p\geq 0\}.
\end{equation}
%Note that $\mathsf P$ consists precisely of the positive elements
%of $H,$ 
%\begin{equation}
%\label{equation:accaplus}
%\mathsf{P}=H^+.
% \end{equation}
%
Let the    unital $\ell$-group $(G,u)$ be defined as
\begin{equation}
\label{equation:ebbene}
(G,u)=(\mathbb Z +_{lex} H, 1),
\end{equation}
where  $+_{lex}$ is lexicographic product and 
  the distinguished unit $u$ is given by the constant
function 1 over $Q$.  For each integer $m$,  $G$ has a copy
of $H$ at level $m$.  In more detail, each element of $G$ is a function
$l\colon Q\to \mathbb R$ of the form  $m+f$, where
$m\in \mathbb Z$ and $f\in H.$  Elements  $l',l''$ of $G$ are acted upon
by pointwise sum and subtraction, and by the {\it lexicographic
lattice order}: thus  $l'\vee l''$ is the pointwise max of
$l' $ and $l''$ in case  $l'(0)=l''(0)$;   if  $l'(0)> l''(0)$, then
$l'\vee l'' = l';$ if $l'(0) <  l''(0)$ then $l'\vee l'' = l''.$ 
By  \eqref{equation:amalg}-\eqref{equation:acca} we have isomorphisms
\begin{equation}
\label{equation:ebbene-bis}
\Gamma(G,u)\cong \McN(\interval^2)/\mathfrak o_{(0,0)}
\cong {\rm C}\amalg{\rm C}.
\end{equation}
 
 \medskip
To conclude the proof it suffices to exhibit a surjectve non-injective
unital 
endomorphism $\sigma$ of $(G,u)$.
Let the map $q$ be defined by
$$
q\colon (x,y)\in Q \mapsto (x,y+x)\in Q.
$$ 
Observe that $q\in G$.
Let the homomorphism $\sigma\colon (G,u)\to (G,u)$ be defined by
$$
\sigma\colon f\in (G,u)\mapsto f\circ q\in (G,u).
$$
A direct  inspection shows that $\sigma$  is surjective:
As a  matter of fact,
the first coordinate function  $x$ on $Q$  is obtainable
as $x\circ q.$ The second coordinate function
$y$ is obtainable as $((y-x)\vee 0)\circ q.$
Thus  $\range(\sigma)=G.$ 
However,  $\sigma$ is not injective: for instance,
 the nonzero function
$k \in G$ given by $k(x,y)=(x-y)\vee 0$ belongs to 
$\ker(\sigma)$.
Thus  $(G,u)$ is not hopfian. From
\eqref{equation:ebbene}-\eqref{equation:ebbene-bis} 
and  Theorem \ref{theorem:gamma}  
it follows that  C$\amalg$C is not hopfian.

\medskip
(iii)  \cite[Corollary 4]{loa}.
\end{proof}

 \medskip
The unital $\ell$-group $(M,1)$ was introduced
in \cite[\S 4]{mun-jfa} as the $\Gamma$ correspondent
of the free countably generated MV-algebra, 
$(M,1)=\Gamma(\McN(\interval^\omega))$.
 By Theorems \ref{theorem:non-hopf}(i)
and  \ref{theorem:gamma} we have:

\begin{corollary}
\label{corollary:nonhopf-group}
$(M,1)$
 is not hopfian. 
\end{corollary}

\medskip 
In Corollary  \ref{corollary:epi}  the hopfian property
of an MV-algebra 
$A$ is shown to be a consequence of  the following
three conditions:
\begin{itemize}
\item[(a)]
$A$  is  semisimple;

\smallskip
\item[(b)]
$A$ is finitely generated;

\smallskip
\item[(c)]
The  finite rank maximal ideals of $A$ are dense
in $\maxspec(A)$.

\end{itemize}
Our next result exhibits  hopfian and non-hopfian
algebras satisfying any  two of these conditions,
together with the negation of the third one.  

\begin{corollary}
\label{corollary:seven}
Each of the following classes of {\em algebras}
(indifferently meaning MV-algebras or unital $\ell$-groups)
contains a hopfian and a non-hopfian member:
\begin{itemize}

\item[($\bar{\rm a}$bc)]   Non-semisimple, finitely generated
 algebras whose maximal ideals
of finite rank are dense.

\item[(a$\bar{\rm b}$c)]  Semisimple 
not finitely generated algebras, 
whose maximal ideals of finite
rank are  dense.

\item[(ab$\bar{\rm c}$)]  Semisimple,
finitely generated algebras where maximals
of finite rank are not dense.
\end{itemize}
\end{corollary}

\begin{proof}  We first deal with MV-algebras.

\smallskip
($\bar{\rm a}$bc) The 
 Chang algebra C
 is a hopfian example,  by Theorem \ref{theorem:germ}(ii).
The free product MV-algebra
 C$\amalg $C is a non-hopfian 
 example, by Theorem \ref{theorem:non-hopf}(ii).

\smallskip
(a$\bar{\rm b}$c)  A hopfian example is given
by  the uncountable atomic
boolean algebra in \cite[Theorem 6]{loa}.
Non-hopfian examples are given by
all countable boolean algebras,
 (Theorem \ref{theorem:non-hopf}(iii)).

\smallskip
(ab$\bar{\rm c}$)  A hopfian example is given by 
 the subalgebra of  $\interval$
generated by   $\pi/4$,
(Theorem \ref{theorem:germ}(i)).
Non-hopfian examples 
are given by the
MV-algebras considered in 
Theorem \ref{theorem:eigen}, notably, the
two-generator
MV-algebra of Figure \ref{figure:eigen-tris}.

\smallskip
For hopfian and non-hopfian examples of unital $\ell$-groups 
one routinely takes the
 $\Gamma$
 correspondents of the above MV-algebras,
in the light of  Theorem \ref{theorem:gamma}. 
\end{proof}

\section{Applications to $\ell$-groups}
\label{section:applications1}
%%%%%%%%%%%%%%%%%%%%%%%%%%%%%%%%

By an {\it $\ell$-group} $G$ we mean a lattice-ordered abelian group.
$G$ does not have a distinguished unit---indeed, $G$ need not
have a unit.
Our main aim in this section is to show that finitely generated
free $\ell$-groups are hopfian and to
characterize their free generating sets.

The  {\it free  $n$-generator
$\ell$-group}  $\mathcal A_n$ consists of all continuous
 piecewise  homogeneous linear functions on $\mathbb R^n$
with integer
coefficients, \cite[Theorem 6.3]{andfei}.
 The maximal spectral space
 of $\mathcal A_n$ is homeomorphic to the $(n-1)$-sphere  $S^{n-1}$. 

   It should be noted that the two categories
  of $\ell$-groups and unital $\ell$-groups have 
  more differences than similarities: 
  While $\ell$-groups
  are definable by equations, the archimedean property
  of the unit in unital $\ell$-groups is not even definable in
  first-order logic. The maximal spectral space of
  an $\ell$-group may be empty, which is never the
  case for unital $\ell$-groups.
 From the Baker-Beynon duality,
 \cite{bak, bey-uno, bey-due}  it easily follows that
 finitely presented
  $\ell$-groups are dual to rational polyhedra
  with piecewise affine linear {\it rational} maps, while finitely presented
  unital $\ell$-groups are dual to rational polyhedra
  with piecewise affine linear {\it integral}  maps. This endows
  the latter dual pair (but not the former) with an interesting
   measure
  \cite[\S 14]{mun11}, 
  and a rich theory of ``states'', \cite[\S 10]{mun11}. 
While the Baker-Beynon duality  also  implies that  finitely
  presented  and  
 finitely generated projective $\ell$-groups coincide, this
  is  far from  true of finitely generated projective
  unital $\ell$-groups: as a matter of fact, their analysis is a tour de force
in algebraic topology,  \cite[\S 17 and references therein]{mun11}.

This said,  the hopfian property of
 finitely generated free $\ell$-groups will be obtained from
the proof that
finitely generated free unital $\ell$-groups and MV-algebras
are hopfian.

%We assume the reader has some acquaintance with the
%Baker-Beynon duality.  

\smallskip 
We  prepare the following lemma:

\begin{lemma} 
\label{lemma:main}
Every
$n$-element generating set of the free $n$-generator
MV-algebra  $\McNn$
is a free generating set.
\end{lemma}

\begin{proof} 
 $\McNn$ comes equipped 
with the free generating set $\{\pi_1,\ldots,\pi_n\}$, where
$\pi_i\colon \cube\to \interval$ is the $i$th coordinate function.
Let
$\{g_1,\ldots,g_n\}$ be a generating set of  $\McNn$. 
Let  the $\mathbb Z$-map $g\colon \interval^n\to\interval^n$
be defined by  
$g(x)=(g_1(x),\ldots,g_n(x))$ for all $x\in \interval^n$.
Let 
$\eta\colon \McNn\to \McNn$
be the endomorphism  of $\McNn$  canonically
extending the map $\pi_j\mapsto g_j,\,\,\,(j=1,\ldots,n).$
Any  $f\in \McNn$
 is mapped by $\eta$ into the
  composite McNaughton function  $f\circ g$. 
 Arguing by way of contradiction,
assume $\{g_1,\ldots,g_n\}$ does not freely generate
$\McNn$.  Then
$\eta$ is not an automorphism of  $\McNn$.
Since 
$\eta$ is onto $\McNn$  then  $\eta$ is not one-one.
%\begin{equation}
%\label{equation:not-one-one}
%\mbox{$\eta$\, is not one-one.}
%\end{equation}

\medskip
\noindent
{\it Claim.} Let $R\subseteq \interval^n$ denote the range of $g$.
Observe that $R$ is a rational polyhedron, \cite[Lemma 3.4]{mun11}.
Let $\McN(R)$ be the MV-algebra 
of restrictions to $R$  of the McNaughton functions of
$\McNn$.
Then there is an isomorphism of  $\McNn$ onto
$ \McN(R)$. 

 \smallskip
Indeed, the
 homomorphism  $\eta' \colon \McN(R)\to \McNn$  defined by
$h \mapsto h\circ g$ for all $h \in \McN(R)$
 is onto  $\McNn$ because the $g_i$ generate
$\McNn.$ Let  $0\not= l\in \McN(R)$, say  $l(y)\not=0$
for some $y\in R.$  There is  $x\in \cube$ with  $y=g(x)$.
Then  $\eta' (l)=l\circ g$ does not vanish at $x$, which shows
that $\eta' $ is one-one.  Our claim is proved.

\smallskip
 Now we note that $R$ is strictly contained in 
 $\interval^n$, for otherwise the
 map $\eta'$  coincides with  
 $\eta$, whence  $\eta$  is  one-one, which is impossible. 
Thus there exists  a point
$z\in \cube\setminus R$. Since $R$ is a closed subset
of $\cube$ we may assume $z\in \mathbb Q^n.$
Let $d=\den(z)$ be the least common denominator of the coordinates
of $z$. 
By Theorem \ref{theorem:semisimple}, the
 maximal ideal $\mathfrak m_z=\{f\in \McN(R)\mid f(z)=0\}$ 
%$x\in R\mapsto \mathfrak m_x
%\in \maxspec({\McN(R)})$ 
%%
%%
%%
%%
%%
has the property that  the quotient
$\McN(R)/\mathfrak m_z$ is the MV-chain with $d+1$ elements.
Conversely, if  $y\in R$ and  $\McN(R)/\mathfrak m_y$ is
the MV-chain with $d+1$ elements then $\den(y)=d$. 
(If necessary, see \cite[Proposition 4.4(iii)]{mun11} for details.)
For  $b=1,2,\dots,$ let  $N_b(\McN(R))$ be the number of maximal ideals  
$\mathfrak m\in \maxspec(\McN(R))$
 such that $\McN(R)/\mathfrak m$ has  
$b+1$ elements.  Then   $N_b(\McN(R))$  coincides   
 with the  (trivially finite)  number of rational points $s$
in $R$ such that $\den(s)=b.$  Let  $N_b(\McNn)$
be similarly defined.   From 
$R \subsetneqq \cube$ and $R$ being closed,  it follows that
$N_b(\McNn)$  is strictly larger
than $N_b(\McN(R))$ for all large $b$. This contradicts
 the existence of the isomorphism  $\eta'$  of
$\McN(R)$ onto $ \McNn$.
\end{proof}

\begin{remark}
A different proof of Lemma
\ref{lemma:main}
can be obtained from  \cite[Remark after Theorem 2]{jontar}
(also see
\cite[Corollary 8]{mal-book}),
 using the
well known fact 
that the variety of MV-algebras is generated by 
its finite members,
\cite[proof of Chang completeness theorem 2.5.3]{cigdotmun}.  
However,  
 the elementary proof given here is of  some interest,
  because, as will be shown in the next result,  it works equally well  for 
  $\ell$-groups,  to which
 \cite{jontar} and  \cite[Corollary 8]{mal-book} 
are not apparently applicable.
 \end{remark}

\begin{corollary}
\label{corollary:main-group}  Any  
generating set $\{g_1,\ldots,g_n\}$ of the free $n$-generator 
$\ell$-group   $\mathcal A_n$ is free generating.
\end{corollary}

\begin{proof}  
%In the  proof of Lemma
%\ref{lemma:main}
%replace the free MV-algebra
%$\McNn$ by the free $\ell$-group   $\mathcal A_n$.
Let   $R'$ be the range  of the map
$g\colon x\in \Rn \mapsto  (g_1(x),\ldots,g_n(x))\in \Rn.$   
An adaptation of the proof of the claim
in  Lemma \ref{lemma:main},
  yields an isomorphism  $h\mapsto h\circ g$ 
between  $\mathcal A_n$  and
 $\mathcal A_n\restrict R'=\{f\restrict R'\mid
f\in \mathcal A_n\}$.
Assuming by way of contradiction
that  $\{g_1,\ldots,g_n\}$  is not free generating, we have a point
$z' \in S^{n-1}\setminus  R'$.
We now compare  the
maximal 
spectral space  $S^{n-1}$ 
of  $\mathcal A_n$  with the maximal spectral
space  $S'$ of  $\mathcal A_n\restrict R'$. The Baker-Beynon
duality  \cite{bak, bey-uno, bey-due}  shows 
that $S'$ is homeomorphic to a subset of
$S^{n-1}\setminus \{z'\}$.  
A contradiction is obtained since
 the $(n-1)$-sphere $S^{n-1}$   is not
homeomorphic to any of its 
proper subsets, \cite[p. 180]{ams}.
\end{proof}

\begin{corollary}
\label{corollary:hopfian-freel}  
(i)  For all $n=1,2,\ldots$, \,\, 
the free $n$-generator 
$\ell$-group  
$\mathcal A_n$  is hopfian.

\smallskip
(ii) Let the ideal $\mathfrak q$  of
$\mathcal A_2$ be given by all functions of
$\mathcal A_2$ identically vanishing over the first quadrant $Q$
of $\mathbb R^2$.
Then the quotient  $\ell$-group 
$\mathcal A_2/\mathfrak q$ is not hopfian.
\end{corollary}

\begin{proof}
(i)  
We let  $\pi_i\colon \Rn\to \Rn$ denote the
$i$the coordinate function. Then $\mathcal A_n$ is
freely generated by the set 
$\{\pi_1,\dots, \pi_n\}$.
%(There is no possibility
%of confusion with other uses of the symbols $\pi_i$
%in this paper.)  
Let $\omega\colon \mathcal A_n\to\mathcal A_n$ be a
surjective homomorphism, with the intent of proving that
$\omega$ is one-one.
The set $E=\{\omega(\pi_1),\dots,
\omega(\pi_n)\}=\{e_1,\ldots,e_n\}$ 
generates $\mathcal A_n$.
%, in symbols,  $\gen(E)=\mathcal A_n$.
By Corollary \ref{corollary:main-group},
$E$ is a free generating set of $\mathcal A_n.$
Let the map $e\colon \Rn \to \Rn$ be defined by
$$
e(x)=(e_1(x),\dots,e_n(x))\mbox{ for all } x=(x_1,\dots,x_n)\in \Rn.
$$
It follows that $\range(e)=\Rn$
and $\omega(g)=g\circ e$
for each  $g\in \mathcal{A}_n$.
If (absurdum hypothesis) 
 there is  a nonzero 
 $ f\in \mathcal A_n$ such that $0= \omega(f)=f\circ e$, let 
$x\in \Rn$ be  such that  $f(x)\not=0$.    
 Since
 $f$ 
identically  vanishes over $\range(e)$,
then  $x$ does not belong to
 $ \range(e)$,  a contradiction showing that $\omega$ is one-one
 and $\mathcal A_n$ is hopfian.

% because  .
%On the one hand, the maximal ideal space 
%$\maxspec(\mathcal A_n)$  
%is homeomorphic to the $(n-1)$-sphere $S^{n-1}$.
%On the other hand,
%the maximal ideal space of $\gen(E)$ is homeomorphic to
%$\range(e)$ which,  in turn,  is homeomorphic to
%$S^{n-1}\setminus X$ for some set
%$X\subseteq \Rn$ containing  $x$.
%This follows from the Baker-Beynon duality.
%Since  (as  already  noted)
%$S^{n-1}$ is not homeomorphic to  $S^{n-1}\setminus X $,
%the  $\ell$-group
% $\gen(E)$ is not isomorphic to $\mathcal A_n$,
% a contradiction. Thus $\omega$ is one-one, and 
% $\mathcal A_n$ is hopfian.
 
\medskip
 
 (ii)  
As in the proof of 
   Theorem \ref{theorem:non-hopf}(ii), let the
   $\ell$-group  $\mathcal A_2\restrict Q$ be defined by
 $\mathcal A_2\restrict Q =
  \{l\restrict Q\mid l \in \mathcal A_2 \}$.
   
\smallskip
  \noindent
  {\it Claim. } 
   Let the function  $q\in\mathcal A_2$ be defined by
  $q(x,y)=0\vee-x\vee -y$.  
Then
 $\mathfrak q$  is the ideal generated by $q$
 in $\mathcal A_2$.  In other words, for all
  $0\leq l\in \mathcal A_2$,  
\begin{equation}
\label{equation:iff}
 \mbox{$l\in\mathfrak q$  iff $mq\geq l$ for 
  some positive integer multiple  $mq$
  of $q$.}
  \end{equation} 
 
 As a matter of fact, 
 if  $mq\geq l$ then trivially $l$ vanishes identically over
 $Q$, because so does $q$. Conversely, if $l=0$ over $Q$
 then let  $a=\partial l(1,0)/\partial y^-$ and
 $b=\partial l(0,1)/\partial x^-$.  For all large integers
 $n$,   $a<\partial nq(1,0)/\partial y^-$ and
 $b<\partial nq(0,1)/\partial x^-$, whence for some
 suitably large  $m\in\mathbb Z$,   $mq$ will be $\geq l$ over
 $\mathbb R^2.$ We have thus proved \eqref{equation:iff}
 and settled our claim.

As a consequence, for any  two  
 functions $f,g\in \mathcal A_2$ we have $f/\mathfrak q=g/\mathfrak q$   
 iff    $|f-g|$ vanishes identically over $Q$ iff
 $f\restrict Q=g\restrict Q$. 
In conclusion, we have an isomorphism
 $
  \mathcal A_2\restrict Q\cong \mathcal A_2/\mathfrak q.
 $ 
In Theorem \ref{theorem:non-hopf}(ii)
it is proved  that $ \mathcal A_2\restrict Q$ is a non-hopfian
{\it unital}  $\ell$-group. A fortiori, 
 $ \mathcal A_2/\mathfrak q$ is a non-hopfian
 $\ell$-group. 
 \end{proof}

%%%%%%%%%%%%%%%%%%%%%%%%%%%%%%%%
%%%%%%%%%%%%%%%%%%%%%%%%%%%%%%%%
\section{Hopficity of the Farey-Stern-Brocot AF C$^*$-algebra, 
\cite{mun-adv}, \cite{boc}, \cite{mun-mjm}}
\label{section:applications2}

Using the results of the earlier sections, in
 this section we will prove that the
 (Farey-Stern-Brocot)  AF C$^*$-algebra  $\mathfrak M_1$
 has the hopfian property and is residually finite dimensional. 
Further, the hopfian property of   $\mathfrak M_1$ extends to all
 its primitive quotients. Owing to its remarkable 
 properties (see Remark \ref{remark:properties}),
\,\,$\mathfrak M_1$\,  has drawn increasing attention
in recent years, \cite{boc, eck, mun-lincei, mun-mjm, nik},
 after a latent
period of over twenty years since its introduction in
\cite{mun-adv}.

We refer to  \cite{dix} for background on C$^*$-algebras,
and to \cite{eff} for 
AF C$^*$-algebras, Elliott classification and $K_0$.
Our pace will be faster than in the previous sections. 

A  {\it unital  AF C$^*$-algebra}  $\mathfrak U$ is the norm closure
of an ascending sequence of finite dimensional C$^*$-algebras,
all with the same unit.
Elliott classification and further $K_0$-theoretic developments  
 \cite{eff} yield a functor
 $$
 \mathfrak U\mapsto K_{0}({\mathfrak U})
 %\quad{\mbox{$\mathfrak U$ an arbitrary   AF C$^*$-algebra}}  
 $$
 from  unital AF $C^{*}$-algebras to
  countable unital dimension groups.
 $K_0$  is an
 order-theoretic refinement of  Grothendieck's
functor.
The Murray-von Neumann order of projections
in $\mathfrak U$ is a lattice iff $K_{0}({\mathfrak U})$
is a unital  $\ell$-group.  $K_0$  sends any
(always closed and two sided) ideal $\mathfrak I$ of an
AF $C^{*}$-algebra $\mathfrak U$ to an ideal
$K_0(\mathfrak I)$ of $K_0(\mathfrak U)$.

As the most elementary  specimen of the free unital $\ell$-groups
  introduced at the end of
Section \ref{section:examples},
 $(M_1,1)$
 denotes   
the unital $\ell$-group of all
continuous piecewise (affine)  linear functions
$f\colon [0,1]\to \mathbb R$, where each
piece of $f$  is a linear polynomial with integer coefficients.
The constant function $1$ on $\interval$  is the distinguished 
 unit of
$(M_1,1)$. 
 
The so called  Farey-Stern-Brocot unital
 AF $C^{*}$-algebra  $\mathfrak M_{1}$
    was originally defined
in \cite[\S 3]{mun-adv} by
\begin{equation}
\label{equation:ko} 
K_0({\mathfrak M_1})=(M_1,1).
\end{equation}
$\mathfrak M_{1}$ was subsequently  rediscovered in 
 \cite{boc}, and denoted $\mathfrak A$. 
Using
 \cite[Theorem 3.3]{mun-adv}, in \cite[Theorem 1.1]{mun-mjm}  
the  Bratteli diagram of $\mathfrak M_{1}$ 
(which is immediately obtainable from 
the matricial presentation \cite[p. 35]{mun-adv}), 
was shown to coincide with the diagram introduced in  
  \cite{boc}, namely:

	\begin{center}
	\unitlength=2.5mm
	\begin{picture}(46,25)(0,0)

		\put(17.6,21){1}
			\put(21.6,21){1}
		\put(18,20){\circle*{0.4}}
		\put(22,20){\circle*{0.4}}
		%%%%%%%%
		
			\put(16,15){\circle*{0.4}}
\put(16,15){\line(2, 5){2}}

			\put(20,15){\circle*{0.4}}
			\put(20,15){\line(-2, 5){2}}
				\put(20,15){\line(2, 5){2}}
				
			\put(24,15){\circle*{0.4}}
				\put(24,15){\line(-2, 5){2}}
					%%%%%%%%
			
			\put(13,10){\circle*{0.4}}
			\put(13,10){\line(3, 5){3}}
			
			\put(17,10){\circle*{0.4}}
\put(17,10){\line(-1, 5){1}}
\put(17,10){\line(3, 5){3}}

					\put(20,10){\circle*{0.4}}
						\put(20,10){\line(0, 5){5}}
						
						\put(23,10){\circle*{0.4}}
\put(23,10){\line(1, 5){1}}
\put(23,10){\line(-3, 5){3}}
						
							\put(27,10){\circle*{0.4}}
							\put(27,10){\line(-3, 5){3}}
								%%%%%%%%
							
\put(10,5){\circle*{0.4}}
	\put(10,5){\line(3, 5){3}}
	\put(10,5){\line(-1, -2){1}}
\put(10,5){\line(1, -2){1}}
	
\put(13,5){\circle*{0.4}}
\put(13,5){\line(0, 5){5}}
\put(13,5){\line(4, 5){4}}
	\put(13,5){\line(-1, -1){2}}
\put(13,5){\line(0, -2){2}}
\put(13,5){\line(1, -1){2}}

	\put(16,5){\circle*{0.4}}
	\put(16,5){\line(1, 5){1}}
	\put(16,5){\line(-1, -2){1.0}}
\put(16,5){\line(0, -2){2}}
\put(16,5){\line(1, -2){1}}
	
	\put(18,5){\circle*{0.4}}
		\put(18,5){\line(-1, 5){1}}
			\put(18,5){\line(2, 5){2}}
	\put(18,5){\line(-1, -2){1}}
\put(18,5){\line(0, -2){2}}
\put(18,5){\line(1, -2){1}}
	
\put(20,5){\circle*{0.4}}
\put(20,5){\line(0, 5){5}}
	\put(20,5){\line(-1, -2){1}}
\put(20,5){\line(0, -2){2}}
\put(20,5){\line(1, -2){1}}

\put(22,5){\circle*{0.4}}
		\put(22,5){\line(-2, 5){2}}
			\put(22,5){\line(1, 5){1}}
	\put(22,5){\line(-1, -2){1}}
\put(22,5){\line(0, -2){2}}
\put(22,5){\line(1, -2){1}}

	\put(24,5){\circle*{0.4}}
		\put(24,5){\line(-1, 5){1}}
	\put(24,5){\line(-1, -2){1}}
\put(24,5){\line(0, -2){2}}
\put(24,5){\line(1, -2){1}}
	
	\put(27,5){\circle*{0.4}}
	\put(27,5){\line(0, 5){5}}
\put(27,5){\line(-4, 5){4}}
	\put(27,5){\line(-1, -1){2}}
\put(27,5){\line(0, -2){2}}
\put(27,5){\line(1, -1){2}}
	
\put(30,5){\circle*{0.4}}
\put(30,5){\line(-3, 5){3}}
\put(30,5){\line(-1,-2){1}}
\put(30,5){\line(1,-2){1}}
\put(8,2){\ldots\,\,\ldots\,\,
\ldots\,\, \ldots\,\, \ldots\,\,
\ldots\,\, \ldots\,\, \ldots\,\, \ldots}

\end{picture}

	\end{center}
	
\noindent {\small The two top  (depth 0) vertices 
%$a$ and $b$
are labeled 1.
%,  $\Lambda(a)=\Lambda(b)=1$.  
The label  
%$\Lambda(v)$ 
of any vertex
$v$ at depth $d=1,2,\ldots, $  is the sum of the labels
of the vertices at depth  $d-1$ connected to $v$ by an edge.
}%%footnotesize

\medskip

From  AF C$^*$-algebraic $K_0$-theory we have:

\begin{corollary} 
\label{corollary:finale}
The map 
$
\mathfrak I\mapsto K_0(\mathfrak I) 
$
is an order isomorphism
  (with respect to  inclusion) 
 of  the set of  ideals    of
$\mathfrak M_1$  onto
the set of  ideals   of $(M_1,1)$. 
%
%$K_0$ thus determines 
% a homeomorphism between the set of primitive ideals of
%$\mathfrak M_1$  and  the set of
%prime  ideals of  $(M_1,1)$  (both sets equipped
%with their respective hull kernel topologies).
%
%
Further, in the category of  unital dimension groups we 
have the isomorphism
 \begin{equation}
\label{equation:sblocco}
K_0\left(\frac{\mathfrak M_1}{\mathfrak I}
\right)\cong\frac{K_0({\mathfrak M_1})}{K_0(\mathfrak I)},
\end{equation}
which is automatically an isomorphism 
in the category of  unital $\ell$-groups.

\end{corollary}
\begin{proof}
The first statement  follows from  Elliott classification
of AF C$^*$-algebras and subsequent   $K_0$-theoretic developments,
 \cite[1.2-1-4]{han}.  The isomorphism
\eqref{equation:sblocco} follows because $K_0$ preserves
short exact sequences, \cite[p. 34]{ell}, \cite[\S 9]{eff}. 
\end{proof}

\begin{remark}
\label{remark:properties}
Other interesting properties of $\mathfrak M_1$ include:
All primitive ideals of $\mathfrak M_1$ are essential,
\cite[Theorem 4.2]{mun-mjm}
(also see \cite[Theorem 8.4]{mun-jfa} for a precursor of this result).
$\mathfrak M_1$ 
has a   faithful invariant tracial state, 
\cite[Theorem 3.1]{mun-lincei}.
Up to isomorphism, the Effros Shen C$^*$-algebras
$\mathfrak F_\theta$ are precisely the infinite dimensional simple
quotients of $\mathfrak M_1$, \cite[Theorem 3.1(i)]{mun-adv}.
The center of 
$\mathfrak M_1$  is
the C$^*$ algebra $\mathsf{C}[0, 1]$ of continuous complex
valued functions on $\interval$,  \cite[p. 976]{boc}. 
The state space of $\mathsf{C}[0, 1]$  is
affinely  weak$^*$ homeomorphic to  
 the space of tracial states on 
$\mathfrak M_1$, \cite[Theorem 4.5]{mun-mjm}.
Every state on the center $\mathsf{C}[0, 1]$ 
of $\mathfrak M_1$  
has a unique
tracial extension to $\mathfrak M_1$,
 \cite[Theorem 2.5]{eck}.
 In \cite[Theorem 21]{cabmun-etds}
 one finds a detailed analysis of the isomorphism classes of
 germinal quotients of $\mathfrak M_1$.   
The automorphism group
of $\mathcal M_1$  has exactly two connected components,
\cite[Theorem 4.3]{mun-mjm}.
In \cite{eck}  the
Gauss map, which is a
Bernoulli shift for continued fractions,
is generalized 
in  the noncommutative setting provided by $\mathfrak M_1$.
%O. Bratteli,
%The center of approximately finite-dimensional C-algebras. J. 
%Funct. Anal. 21.2 (1976) 195Ð202.
\end{remark}

% \item[(iii)] 
%$ (M_1,1)$ is the unital $\ell$-group of all
%continuous piecewise linear functions
%$f\colon [0,1]\to \mathbb R$, where each
%piece of $f$  is a linear polynomial with integer coefficients.
%The constant function 1 is the distinguished order unit of
%(M_1,1)$. 
 
% \smallskip
% 
%(ii) Applying Theorem \ref{theorem:semisimple} to the
%MV-algebra $\McN(\interval)=\Gamma(M_1,1)$ we
%obtain homeomorphisms
%$$
%\interval \cong \maxspec(\McN(\interval))\cong\maxspec(M_1,1)\cong
%\mbox{the   maximal ideal 
%space of  $\mathfrak M_1$}.
%$$
%% then get
%%a homeomorphism  
%%$y \mapsto \mathfrak N_y$
%%between 
%%the unit real interval $[0,1]$ and the   maximal ideal 
%%space of  $\mathfrak M_1$ equipped with the hull-kernel topology.

By  Theorem \ref{theorem:gamma},
the unital $\ell$-group
 $(M_1,1)$  inherits via $\Gamma$ the semisimplicity of 
 the free MV-algebra $\McN(\interval)$.
Finite rank  maximal ideals of    $(M_1,1)$ 
correspond to 
 rational points in $\interval$ via  the
 homeomorphism of Theorem \ref{theorem:semisimple}.
 By Theorem \ref{theorem:residually-finite},
 semisimplicity and a dense set of finite rank
 maximal ideals are the counterpart for 
 $(M_1,1)$ of the
 residual finiteness of    $\McN(\interval).$

For the AF C$^*$-algebraic
  counterpart  of 
the residual finiteness of  the MV-algebra
 $\McN(\interval)=\Gamma(M_1,1)=\Gamma(K_0(\mathfrak M_1))$,
 recall  that a C$^*$-algebra $A$ is
 {\it residually finite-dimensional} 
 if it has a separating family of finite dimensional representations.
 % \cite{goomen}.
In other words, 
 for any nonzero element $a\in A$  there is a
finite dimensional representation  $\pi$  of $A$ 
with $\pi(a)\not=0.$

\begin{theorem}
\label{theorem:res-fin-dim}
$\mathfrak M_1$
 is residually finite-dimensional
 \end{theorem}

\begin{proof}
Following \cite{mun-jfa}, 
let   $\Maxspec(\mathfrak M_1)$ denote the set of
maximal ideals of $\mathfrak M_1$ with the hull kernel  
topology inherited from ${\rm prim}(\mathfrak M_1)$
by restriction.  ($\Maxspec$ is denoted ${\rm MaxPrim}$ in \cite{mun-mjm}.)  
Since $(M_1,1)$ is semisimple, 
the order isomorphism
\begin{equation}
\label{equation:order}
 \mathfrak N\in \Maxspec(\mathfrak M_1)
\mapsto K_0(\mathfrak N)\in \maxspec(M_1,1)
\end{equation}
 of  Corollary \ref{corollary:finale} 
entails
\begin{equation}
\label{equation:semisimple*}
\bigcap \Maxspec(\mathfrak M_1)=\{0\}.
\end{equation} 
By definition of the respective
 topologies of the maximal spectral spaces of
$\mathfrak M_1$ and of $(M_1,1)$, the order isomorphism
\eqref{equation:order}
is also  a homeomorphism
  $$
  K_0\colon  \Maxspec(\mathfrak M_1)\cong
  \maxspec(M_1,1). 
  $$
  By Theorem \ref{theorem:gamma}, the map
  $$
  \mathfrak n\in \maxspec(M_1,1)\mapsto \mathfrak n\cap 
  \interval
  \in \maxspec(\McN(\interval))
  $$
  is a homeomorphism of $\maxspec(M_1,1)$ onto 
  $ \maxspec(\McN(\interval))$.
By Theorem \ref{theorem:semisimple}(vi),  the map
$$
\xi\in \interval \mapsto \mathfrak m_\xi\in \maxspec(\McN(\interval))
$$ 
is a homeomorphism of the unit real interval
$\interval$ onto the maximal spectral space
$\maxspec(\McN(\interval))$,
(see \cite[Lemma 11]{boc} or \cite[(6) and Corollary 3.4]{mun-mjm} for details).
From  these homeomorphisms we obtain
a homeomorphism
\begin{equation}
\label{equation:homeo}
\rho\in \interval \mapsto \mathfrak N_\rho\in 
\Maxspec(\mathfrak M_1)
\end{equation}
having the following property:
For any rational $p/q\in \interval$, the quotient
$
\mathfrak M_1/\mathfrak N_{p/q} 
$  
is  the C$^*$-algebra  $\mathsf M_q$ of $q\times q$ complex matrices.  
See \cite[Proposition 4]{boc},
(also see \cite[Corollary 3.3(ii)]{mun-mjm}).
By \cite[Theorem 3.1(i)]{mun-adv},
 for every irrational $\theta\in \interval$   the quotient C$^*$-algebra
$
\mathfrak M_1/\mathfrak N_{\theta} 
$
coincides
 with the Effros-Shen algebra  $\mathfrak F_\theta$, \cite{eff}.  

Let $0\not=a\in \mathfrak M_1$. 
By  \eqref{equation:semisimple*} there is $\mathfrak N\in 
\Maxspec(\mathfrak M_1)$ such that $a/\mathfrak N\not=0,$
i.e., $a\notin \mathfrak N.$
By definition of hull-kernel topology,
 the set of maximal ideals $\mathfrak N$ of $\mathfrak M_1$ with
$a/\mathfrak N\not=0$ is open in
$\Maxspec(\mathfrak M_1)$. Thus by  
 \eqref{equation:homeo} 
 we have  a rational $s/t\in \interval$
such that  $a/\mathfrak N_{s/t}\not=0.$
The quotient map  $\pi\colon \mathfrak M_1
\to \mathfrak M_1/\mathfrak N_{s/t}\cong \mathsf M_t$ 
yields  the desired finite dimensional
representation with $\pi(a)\not=0.$   
\end{proof}

We next show that
the unital AF C$^*$-algebra
$\mathfrak M_1$ inherits the hopfian
property from its associated unital
$\ell$-group $(M_1,1)$:

\begin{theorem} 
\label{theorem:af-hopf}
 Suppose we are given  a short
 exact sequence
 $$
 0 \xrightarrow{} \mathfrak I \xrightarrow{\iota} \mathfrak M_1
\xrightarrow{\sigma}\mathfrak M_1/ \mathfrak I \xrightarrow{} 0
 $$
 where $\mathfrak I$ is 
 an  ideal of  $\mathfrak M_1$,
 \,\,\,$\iota$ is the identity map on $\mathfrak I\subseteq
 \mathfrak M_1$, and
 $\sigma\colon \mathfrak M_1
 \to \mathfrak M_1/\mathfrak I$ is the quotient map, with
 $K_0(\sigma)$
a   homomorphism   of $K_0(\mathfrak M_1)$
into $K_0(\mathfrak M_1/\mathfrak I)$ in the category
of unital 
 $\ell$-groups. 
 Suppose we have an isomorphism
  $\theta\colon\mathfrak M_1/\mathfrak I\cong 
 \mathfrak M_1$. Then
 $\theta\circ\sigma$ is an automorphism of
 $\mathfrak M_1$.

%($K_0(\sigma)\colon K_0(\mathfrak M_1)\mapsto 
%K_0(\mathfrak M_1/\mathfrak I)$ is a surjection,
%in the category of unital dimension groups and)

 \end{theorem}

\begin{proof} 
It is well known that
  $\mathfrak I$ is an AF C$^*$-algebra
and $\mathfrak M_1/\mathfrak I$ is a unital 
AF C$^*$-algebra, \cite{eff}, \cite[p. 34]{ell}.
 By a special case of the Bott Periodicity Theorem,
 \cite[Corollary 9.2]{eff}, 
the natural map 
\begin{equation}
\label{equation:natural}
K_0(\sigma)\colon K_0(\mathfrak M_1)\to
K_0(\mathfrak M_1/\mathfrak I)
\end{equation}
 is a {\it  surjective} homomorphism   
in the category of unital $\ell$-groups.  By 
 Corollary \ref{corollary:finale},
letting the ideal $\mathfrak i$ of
 $(M_1,1)$ be defined by 
 $\mathfrak i=K_0(\mathfrak I),$
 we may rewrite \eqref{equation:natural}
 as follows:   
$$
K_0(\sigma)\colon
(M_1,1)  \to \frac{(M_1,1)}{\mathfrak i}.
$$
 From the assumed isomorphism
 $$
 \theta\colon\mathfrak M_1/\mathfrak I\cong 
 \mathfrak M_1 
 $$
 
 \smallskip
 \noindent
the functorial properties of $K_0$ yield an isomorphism  
of unital $\ell$-groups

\begin{equation}
\label{equation:nox}
 K_0(\theta)\colon  K_0\left(\frac{\mathfrak M_1}{\mathfrak I}\right) \cong 
 K_0( \mathfrak M_1).
\end{equation}
 Since by \eqref{equation:sblocco},
  $$
 K_0\left(\frac{\mathfrak M_1}{\mathfrak I}\right) \cong
\frac{K_0(\mathfrak M_1)}{K_0(\mathfrak I)}=
\frac{(M_1,1)}{\mathfrak i},
 $$
 
 \smallskip
 \noindent
then 
 \eqref{equation:ko} and \eqref{equation:nox}
 yield an isomorphism
$$
K_0(\theta)\colon
\frac{(M_1,1)}{\mathfrak i}\cong 
(M_1,1)
$$    of unital dimension groups, which is automatically
 an isomorphism of  unital $\ell$-groups.
By our assumption on $K_0(\sigma)$,
 the composite map
$K_0(\theta\circ\sigma)=K_0(\theta)\circ K_0(\sigma)$ is a  
 homomorphism 
of 
$(M_1,1)$ onto   $(M_1,1)$ in the category of unital
$\ell$-groups.
By  Corollary
\ref{corollary:free-group},   $(M_1,1)$
is hopfian, and hence  
$K_0(\theta\circ\sigma)$ is an automorphism of
$(M_1,1)=K_0(\mathfrak M_1)$.
So both $K_0(\theta)$ and  $K_0(\sigma)$ are injective,
whence, by definition of $K_0$, so are  $\sigma$,
and $\theta\circ\sigma.$ It follows that
$\theta\circ\sigma$ is an automorphism of   $\mathfrak M_1.$
\end{proof}

\medskip
\begin{corollary}
\label{corollary:quotients}
For every primitive ideal  $\mathfrak P$ of $\mathfrak M_1$,
 the quotient
 C$^*$-algebra $\mathfrak M_1/\mathfrak P$ inherits the hopfian
property of $\mathfrak M_1$  of Theorem
\ref{theorem:af-hopf}.
This holds in particular for the Effros-Shen
C$^*$-algebras $\mathfrak F_\theta$ of \cite[\S 10]{eff} and 
for the Behnke-Leptin  C$^*$-algebras  $\mathcal A_{k,q}$  
of  \cite[pp. 330-331]{behlep}.
\end{corollary}

\begin{proof}
Both $\mathfrak F_\theta$  and
$\mathcal A_{k,q}$  arise as primitive
quotients of $\mathfrak M_1$ and have only
a  finite    number  of primitive
quotients. The remaining primitive quotients
of $\mathfrak M_1$ are finite dimensional.
(See \cite[Theorem 3.1(i)]{mun-adv} and \cite[Corollary 3.3]{mun-mjm} for details.)
By Corollary \ref{corollary:finale},
the  unital dimension group  
$(G,u)=K_0(\mathfrak M_1/\mathfrak P)$
 is a (totally ordered) finitely generated unital $\ell$-group
with a finite number ($\in \{1,2\}$) of prime ideals. By 
Corollary  \ref{corollary:germ-group}(ii), $(G,u)$ is
hopfian.  Going backwards via $K_0$ as in the
proof of Theorem \ref{theorem:af-hopf}, we conclude
that  both $\mathfrak F_\theta$ and $\mathcal A_{k,q}$ 
have the desired hopfian property. 
\end{proof}

%%%%%%%%%%%%%%%%%%%%%%%%%%%%
%%%%%%%%%%%%%%%%%%%%%%%%%%%%
%%%%%%%%%%%%%%%%%%%%%%%%%%%% 
%%%%%%%%%%%%%%%%%%%%%%%%%%%%  
\section{Appendix: Background on MV-algebras}
\label{section:appendix}
%%%%%%%%%%%%%%%%%%%%%%%%%%%% 
%%%%%%%%%%%%%%%%%%%%%%%%%%%% 
%%%%%%%%%%%%%%%%%%%%%%%%%%%% 

Here we collect a number of
basic  MV-algebraic results which have been repeatedly used
in the previous sections.
Each  result  comes with a   reference   where
the interested reader can find  a proof.

\subsection{\it Representation}
\label{section:representation}

\begin{lemma}
\cite[1.2.8]{cigdotmun}
\label{1.2.8}
Let $A$ and $B$ be MV-algebras. If $\sigma$
is a homomorphism of $A$ onto $B$  
then there is an isomorphism $\tau$ 
of $A/\ker(\sigma)$ onto $B$
such that  $\tau(x/\ker(\sigma)) = \sigma(x)$ 
for all  $x\in A$.
\end{lemma}

\begin{theorem}
\label{theorem:representation}
(i)
\cite[Theorem 3.5.1]{cigdotmun} 
\label{theorem:simple}
An  MV-algebra  $A$    is simple iff it 
 is isomorphic 
to a subalgebra of the standard MV-algebra
$[0,1]$.

\medskip

(ii) 
   \cite[Theorem 9.1.5]{cigdotmun} 
 \label{theorem:mcngenfunction}
For each cardinal $\kappa$,
the  free MV-algebra $\it Free_\kappa$ on $\kappa$ 
free generators
is given by the {\em McNaughton functions
over   ${\rm [0,1]}^\kappa$,} 
with pointwise operations, i.e.,
those functions 
$g : [0,1]^{\kappa} \rightarrow [0,1]$ 
such that there are ordinals
$\alpha(0) < \ldots  < \alpha(m-1)  < \kappa$  
and   a McNaughton function  $f$
over  $[0,1]^m$  having the following property:
for each   ${ x} \in  [0,1]^{\kappa}$,
$g({ x}) = 
f(x_{\alpha(0)},\ldots, x_{\alpha(m-1)})$.

\medskip

(iii)  
\cite[Theorem 3.6.7]{cigdotmun}  
An MV-algebra  $A$ with
$\kappa$  generators  is 
semisimple iff 
for some nonempty closed subset
$X \subseteq [0,1]^\kappa$,
$A$  is isomorphic to the
MV-algebra  $\McN(X)$  of restrictions
to $X$ of all functions in
$Free_\kappa$.

\end{theorem}

\medskip 
\subsection{\it Yosida duality}
\label{section:Yosida}

%%%%%%%4.15
%\begin{proposition} 
%{\rm  (\cite[Proposition 4.15]{mun11})} 
%  \label{proposition:maxspec}
%For any MV-algebra
%$A$, the maximal spectral space
%$\maxspec(A)$ is a non\-em\-p\-ty compact Hausdorff space.  
%\end{proposition}

%%%%%%4.16
For any nonempty compact
Hausdorff space $X\not=\emptyset$
we let $C(X)$
denote the MV-algebra of all continuous $[0,1]$-valued
functions on $X$, with the pointwise
operations of the  MV-algebra $[0,1]$.

An MV-subalgebra $A$\ of $C(X)$\ is said to be  
{\em separating\/}
if  for any two distinct points 
$x,y\in X$, there is 
$f \in A$\ such that 
$f(x) = 0$\ and $f(y) > 0$.
Following \cite{mun11}, for
each ideal $\mathfrak i$ of $A$ we let
$$
{\mathcal Z_\mathfrak i}=\bigcap\{f^{-1}(0)\mid f\in \mathfrak i\}.
$$
($\mathcal Z_\mathfrak i$ is denoted $V_\mathfrak i$
in \cite{cigdotmun}.)

\begin{proposition}
\cite[Proposition 3.4.5]{cigdotmun}
\label{proposition:3.4.5}
 Let $X$\ be 
a compact Hausdorff 
space and $A$\ be a separating 
subalgebra of $Cont (X)$. 
Then the map 
$f/\mathfrak i \mapsto f\restrict{\mathcal Z_\mathfrak i}$\ is 
an isomorphism of $A/\mathfrak i$\ onto 
$A\restrict {\mathcal Z_\mathfrak i}$\ 
if and only if  \,$\mathfrak i$\ is 
an intersection of maximal ideals of $A$.
\end{proposition}

For every MV-algebra $A$, we let
$\hom(A)$ denote the set of homomorphisms
of $A$ into the standard MV-algebra $[0,1]$.

%%%%3.4.3 
%\begin{theorem}
%\cite[Theorem 3.4.3]{cigdotmun} 
%\label{genmax=points}
%Let $X$\ be a compact Hausdorff space 
%and $A$\ be a separating 
%subalgebra of the MV-algebra $C(X)$
%of all $[0,1]$-valued continuous functions on $X$.  
% Then the map $x \mapsto \mathfrak m_x
%= \{f\in A\mid f(x)=0\}$\ is a one-one
%correspondence between  $X$\ and  
%${\maxspec} (A)$.
%\end{theorem}

\begin{theorem}
 {\rm (Yosida duality, \cite[Theorem 4.16]{mun11})}
\label{theorem:semisimple}  Let $A$ be an MV-algebra.

\medskip
$(i)$ For any maximal ideal $\mathfrak m $ of $A$ there is a unique
pair $(\overline{\mathfrak m}, I_{\mathfrak m})$ with
$I_{\mathfrak m}$ an MV-subalgebra of $\interval$ and 
 $\overline{\mathfrak m}$  an isomorphism of   $A/\mathfrak m$
onto $I_{\mathfrak m}$.

\medskip
(ii)  The map $\ker\colon \eta \mapsto \ker \eta$
 is a one-one correspondence between $\hom(A)$ and
$\maxspec(A)$.  The inverse map sends each 
$\mathfrak m\in \maxspec(A)$ to the  homomorphism
$\eta_{\mathfrak m}\colon A\to \interval$ given by
$a \mapsto \overline{\mathfrak m}(a/\mathfrak m)$.
For each  $\theta\in \hom(A)$ and $a\in A$,
%\begin{equation}
%\label{equation:abete}
$\theta(a)=\overline{\ker\theta}\left({a}/{\ker\theta}\right).$
%\end{equation}

\medskip
  $(iii)$ 
 The map $^*\colon  a \in A
\mapsto a^* \in [0,1]^{{\maxspec}(A)}$ defined by 
$a^*(\mathfrak m) =
\overline{\mathfrak m}(a/\mathfrak m),$ 
is a homomorphism of $A$ onto
a separating MV-subalgebra $A^{*}$ of $C({\maxspec}(A))$.  
%The basic closed sets of
%${\maxspec}(A)$ have the form 
%$ 
%\mathsf{F}_b = \{{\mathfrak m} \in {\maxspec}(A) \mid
%b^*({\mathfrak m}) = 0\}, 
%$
%where $b$ ranges over all
%elements of $A$.
The map $a\mapsto a^*$ is an isomorphism of $A$
onto $A^*$
iff $A$ is semisimple.

\medskip
(iv) \,Suppose $X\not=\emptyset$ 
is a compact Hausdorff space and $B$ is a
separating subalgebra of $C(X)$.  Then the map 
$\iota\colon x\in
X\mapsto \mathfrak m_x =
\{f\in B \mid f(x)=0\}$
is a homeomorphism of $X$ onto
${\maxspec}(B)$.  The inverse map $\iota^{-1}$  sends each
$\mathfrak m\in \maxspec(B)$ to the only element
  of the set
${\mathcal Z}\mathfrak m$.

\medskip
(v)  From the hypotheses of (iv) it follows that
$
f^*\circ\iota= f
$
 for each $f\in B.$
Thus the map  $f^*\in B^*\mapsto f^*\circ\iota\in C(X)$
is the inverse of the isomorphism  $^{*}\colon B\cong B^*$
defined in (iii). In particular,  $f(x)=f^*(\mathfrak m_{x})$
for each  $x\in X$.

\medskip
(vi)     \cite[Corollary 4.18]{mun11} 
   % \label{corollary:maxspec} 
    For every nonempty closed
subset $Y$ of $\interval^\kappa$, the map 
$\iota \colon x \in Y \mapsto \mathfrak m_{x}
=\{f \in \McN(Y) \mid f(x)=0\}$ of  (iv) 
is a homeomorphism of $Y $ onto
${\maxspec}(\McN(Y))$.  The inverse map $\mathfrak m\mapsto
x_{\mathfrak m}$ sends every maximal ideal 
$\mathfrak m$ of $\McN(Y)$
to the only element   of ${\mathcal
Z}\mathfrak m$. 
\end{theorem}

 %%%%%%%%%%%%%4.17, 4.18
%\begin{remark}
 
 \noindent
 {\it Notational Stipulations.\,\,}
 In the light of 
 Theorem \ref{theorem:semisimple}(i)-(iii), for 
  every MV-algebra  $A$, 
$a\in A$ and $\mathfrak m\in \maxspec(A)$
we will tacitly identify  
$a/\mathfrak m$  with the  real
number  $\overline{\mathfrak m}(a/\mathfrak m)$,
and write 
%
%\begin{equation} 
%    \label{equation:unique-isomorphism} 
$    a/\mathfrak m  =
\overline{\mathfrak m}(a/\mathfrak m)=
a^*(\mathfrak m). 
$
%\end{equation}
%For $\theta\in \hom(A)$ and $a\in A$ we will also write
%%\begin{equation}
%%    \label{equation:serve-per-tensor}
%  $  \theta(a)= {a}/{\ker \theta}.$
    %\end{equation}
%% 
 % Further, when $A$ is semisimple 
 % we will freely
 % use the identification
 % \begin{equation} 
 %     \label{equation:unique-isomorphism}
 % a/\mathfrak m = a^*(\mathfrak m).
 % \end{equation} 
 %%
Further, if 
$B$ is a separating MV-subalgebra
of $C(X)$ as in 
 {{}Theorem 
\ref{theorem:semisimple}(iv)-(v)},
identifying 
    $B$ with $B^*$  and
    $X$ with $\maxspec(B)$, we will  
    write without fear of ambiguity, 
%\begin{equation}
%    \label{equation:hion}
  $  f(x)=f(\mathfrak m_x) = f/\mathfrak m_x$
%\end{equation}
    for each  $x\in X$  and $f\in B$.

\bigskip
   \subsection{\it The $\Gamma$ functor}
    \label{section:gamma}
    
    \begin{theorem}
    \label{theorem:gamma}
  (i)     \cite[Theorem 3.9]{mun-jfa}.
  For each
    unital $\ell$-group  $(G,u)$ let $\Gamma(G,u)$  
    be  the MV-algebra
    $([0,1],0,\neg,\oplus)$ where  $\neg x=u-x$
    and $x\oplus y=\min(u,x+y).$  Further, for every
     homomorphism
    $\xi\colon (G,u)\to(H,v)$ let  
    $\Gamma(\xi)$  be the restriction of $\xi$
    to $[0,u]$.  Then  $\Gamma$ is a categorical
    equivalence between unital $\ell$-groups and
    MV-algebras.

\medskip
(ii) \cite[Theorem 7.2.2]{cigdotmun}. 
Let   $A =\Gamma({G},{u})$. Then the  
correspondence     
$
\phi\colon \mathfrak i \mapsto\phi (\mathfrak i)
= \{x\in G\, \mid \, |x| \wedge u \in \mathfrak i \}
$
\noindent is an order-isomorphism 
from the set 
%${\mathcal I}(A)$ 
of ideals of $A$ onto the set of ideals of $G$,
both sets being ordered by inclusion.
 The inverse 
isomorphism $\psi$ is given by     
$\psi\colon \mathfrak j  \mapsto \psi (\mathfrak j) 
= \mathfrak j \cap [0,u]$.  

\medskip
(iii)   \cite[Theorem 7.2.4]{cigdotmun}.
 For every ideal
$\mathfrak j$\  of $G,$ \,\,\,
$
\Gamma({G/\mathfrak j},{u/\mathfrak j})
$ is isomorphic to the quotient MV-algebra
$
 \Gamma({G},{u})/(\mathfrak j \cap [0,u]).
$
 \end{theorem}

 \bibliographystyle{plain}
%%%%%%%%%%%%%

 \end{document}